\documentclass[reqno]{amsart}

\usepackage[makeroom]{cancel}  
\usepackage{amssymb, color}
\usepackage{mathrsfs}
\usepackage{amsmath}
\usepackage{amsthm}
\usepackage{url}
\usepackage{tikz}
\usepackage{stackrel} 
\usepackage{hyperref}
\usepackage{courier}
\usepackage[all,cmtip]{xy} 
\usepackage{multicol}
\usepackage{mathtools}
\usepackage{verbatim} 

\catcode`~=11 
\newcommand{\urltilde}{\kern -.15em\lower .7ex\hbox{~}\kern .04em}  
\catcode`~=13 

\newcommand{\rk}{\mathrm rk}

\newcommand{\sm}{\mathop{\mathrm{sm}}\nolimits}

\newcommand{\lie}{\mathop{\mathrm{Lie}}\nolimits}

\newcommand{\pr}{\mathop{\mathrm{pr}}\nolimits}

\newcommand{\Sp}{\mathop{\mathrm{Sp}}\nolimits}

\newcommand{\Hilb}{\mathop{\mathrm{Hilb}}\nolimits}

\newcommand{\ind}{\mathop{\mathrm{Ind}}\nolimits}

\newcommand{\Proj}{\mathop{\mathrm{Proj}}\nolimits}

\newcommand{\Ad}{\mathop{\mathrm{Ad}}\nolimits}
\newcommand{\ad}{\mathop{\mathrm{ad}}\nolimits}

\newcommand{\im}{\mathop{\mathrm{Im}}\nolimits}

\newcommand{\spec}{\mathop{\mathrm{Spec}}\nolimits}  

\newcommand{\tr}{\mathop{\mathrm{tr}}\nolimits}

\newcommand{\Sym}{\mathop{\mathrm{Sym}}\nolimits}

\newcommand{\Der}{\mathop{\mathrm{Der}}\nolimits} 
\newcommand{\Id}{\mathop{\mathrm{Id}}\nolimits} 
\newcommand{\Red}{\mathop{\mathrm{Red}}\nolimits} 
\newcommand{\Vect}{\mathop{\mathrm{Vect}}\nolimits}

\newcommand{\gr}{\mathop{\mathrm{gr}}\nolimits}

\newcommand{\Hom}{\mathop{\mathrm{Hom}}\nolimits}

\makeatletter
\newif\if@borderstar
   \def\bordermatrix{\@ifnextchar*{%
       \@borderstartrue\@bordermatrix@i}{\@borderstarfalse\@bordermatrix@i*}%
   }
   \def\@bordermatrix@i*{\@ifnextchar[{\@bordermatrix@ii}{\@bordermatrix@ii[()]}}
   \def\@bordermatrix@ii[#1]#2{%
   \begingroup
     \m@th\@tempdima8.75\p@\setbox\z@\vbox{%
       \def\cr{\crcr\noalign{\kern 2\p@\global\let\cr\endline }}%
       \ialign {$##$\hfil\kern 2\p@\kern\@tempdima & \thinspace %
       \hfil $##$\hfil && \quad\hfil $##$\hfil\crcr\omit\strut %
       \hfil\crcr\noalign{\kern -\baselineskip}#2\crcr\omit %
       \strut\cr}}%
     \setbox\tw@\vbox{\unvcopy\z@\global\setbox\@ne\lastbox}%
     \setbox\tw@\hbox{\unhbox\@ne\unskip\global\setbox\@ne\lastbox}%
     \setbox\tw@\hbox{%
       $\kern\wd\@ne\kern -\@tempdima\left\@firstoftwo#1%
         \if@borderstar\kern2pt\else\kern -\wd\@ne\fi%
       \global\setbox\@ne\vbox{\box\@ne\if@borderstar\else\kern 2\p@\fi}%
       \vcenter{\if@borderstar\else\kern -\ht\@ne\fi%
         \unvbox\z@\kern-\if@borderstar2\fi\baselineskip}%
         \if@borderstar\kern-2\@tempdima\kern2\p@\else\,\fi\right\@secondoftwo#1 $%
     }\null \;\vbox{\kern\ht\@ne\box\tw@}%
   \endgroup
   }
\makeatother

\newtheorem{theorem}{Theorem}[section] 
\newtheorem{proposition}[theorem]{Proposition}
\newtheorem{lemma}[theorem]{Lemma}

\newtheorem{example}[theorem]{Example}
\newtheorem{definition}[theorem]{Definition}
\newtheorem{remark}[theorem]{Remark} 
\newtheorem{conjecture}[theorem]{Conjecture} 
 
\newtheorem{corollary}[theorem]{Corollary}

\newcommand*{\DashedArrow}[1][]{\mathbin{\tikz [baseline=-0.25ex,-latex, dashed,#1] \draw [#1] (0pt,0.5ex) -- (1.3em,0.5ex);}}

\newcommand{\Stab}{\mathop{\mathrm{Stab}}\nolimits}
\newcommand{\I}{\mathop{\mathrm{I}_n}\nolimits}

\begin{document}

\title[Unipotent invariants and the isospectral Hilbert scheme]{Unipotent invariants of filtered representations of quivers and the isospectral Hilbert scheme} 
\author{Mee Seong Im \and Lisa M. Jones}
\address{Department of Mathematical Sciences, United States Military Academy, West Point, NY 10996 USA}
\email{meeseongim@gmail.com}
\address{Department of Mathematical Sciences, United States Military Academy, West Point, NY 10996 USA}
\email{lisa.jones@elmjay.com} 
\address{Department of Pure Mathematics, University of Cambridge, Cambridge, CB3 0WB UK}
\email{lisa.jones@elmjay.com} 
\date{\today}

\begin{abstract}
Given any finite quiver, we consider a complete flag of vector spaces over each vertex. 
Consider the unipotent invariant subalgebra of the coordinate ring of the filtered quiver representation subspace. 
We prove that the dimension of the algebraic variety of the unipotent invariant subalgebra is finite. 
We also construct an ADHM analog for the Borel subalgebra setting, showing its birationality to the isospectral Hilbert scheme. 
Quiver-graded Steinberg varieties, quantum Hamiltonian reduction, and deformation quantization constructions for the nonreductive setting are discussed, ending with open problems. 
\end{abstract}  
\thanks{Both authors acknowledge Frank D. Grosshans for helpful email communications, and the first author thanks  2016 Knot homologies, Hilbert schemes, and Cherednik algebras workshop at the University of Oregon for discussions on the geometry of the varieties discussed in this manuscript, with a special acknowledgment to Ben Elias, Jacob Rasmussen, Eugene Gorsky, Matt Hogancamp, and Paul Wedrich. The first author is supported by Applications in Computational Engineering Mathematics through Mathematical Sciences Center of Excellence (MSCE) at West Point, NY}


\maketitle 

\bibliographystyle{amsalpha} 

\setcounter{tocdepth}{0} 


\section{Introduction}\label{section:intro}  
Parabolic group actions arise naturally in mathematics. For instance, let $B$ be a Borel subgroup of $G = GL_n(\mathbb{C})$ and let $\mathfrak{b}=\lie(B)\subseteq \mathfrak{g} =\lie(G)$ be the Lie algebra of $B$, where $\mathfrak{g}$ is the set of all $n\times n$ matrices over the complex numbers. One could ask to describe the structure of the $B$-orbits on $\mathfrak{b}$, or equivalently, one may be interested in studying the $B$-adjoint action on a complete filtration of an $n$-dimensional complex vector space $V$. 
 
One motivation for our investigation is the connection between $B$-equivariant geometry on $\mathfrak{b}$ and $G$-orbits on the Grothendieck-Springer resolution $\widetilde{\mathfrak{g}}\twoheadrightarrow \mathfrak{g}$. T. Nevins in \cite{Nevins-GSresolutions} (Section 3) shows the isomorphism  between 
$\mathfrak{b}/B\cong \mathfrak{p}/P$ and $\widetilde{\mathfrak{g}}/G$ in terms of Hamiltonian reduction of a parabolic group $P$ acting on $\mathfrak{p}\times V$, where $\mathfrak{p}=\lie(P)$, and of $G\times P$ acting on $G\times \mathfrak{p}\times V$. 
Furthermore, our filtered quiver representations arise as fibers of universal quiver Grassmannians onto quiver Grassmannians or more generally, universal quiver flag varieties onto quiver flags, where the fibers of these projections are homogeneous vector bundles. In fact, M. Reineke in \cite{MR3102955} proved that all projective varieties are quiver Grassmannians. 
This result directly ties into the role that subspaces of quiver representations produce interesting geometric spaces beyond the classical algebraic geometry setting.

Due to the rich geometry interplaying among Hilbert schemes of compactified Jacobians (cf. \cite{Oblomkov-Rasmussen-Shende-HilbScheme-HOMFLY}), 
the representation theory of Cherednik algebras (cf.   \cite{MR3336085}, \cite{MR2210660}, \cite{MR2827838}, \cite{MR2183255}, \cite{MR2219255}),  
Hochschild (co)homology of Soergel bimodules (cf. \cite{MR2339573}), 
the categorification of quiver  
Hecke (Khovanov-Lauda-Rouquier) algebras (cf. \cite{MR2525917}, \cite{MR2763732}, \cite{Khovanov-Lauda-III}, \cite{MR2908731}, \cite{MR2837011}),   to name a few, 
the study of the geometry of the $B$-orbits on $\mathfrak{b}$ are of significant interest in and of itself.  
We also refer the reader to 
\cite{MR1834739}, \cite{MR2838836}, \cite{Ginzburg-Nakajima-quivers}, and Section~\ref{section:geometry} for some background on quivers and the geometry of the (Grothendieck-)Springer resolution. 

We begin the study of the geometry of parabolic subgroups on their subalgebras by generalizing Theorem 1.1 in \cite{Im-filtered-semi-invariants}. Before we restate this theorem, we need to give two definitions. 
Given a quiver $Q$, a quiver path $p$ is a concatenation of arrows in $Q$. If $p$ is a cycle, then $p^m$ is the path composed with itself $m$ times, and we say $p$ is {\em reduced} if $[p]\not=0$ in $\mathbb{C}Q/\langle q^2: q\in \mathbb{C}Q, l(q)\geq 1\rangle$. 
Finally, we define a {\em pathway}  
from vertex $i$ to vertex $j$ as a reduced path from $i$ to $j$.  
Theorem 1.1 in \cite{Im-filtered-semi-invariants} states that a quiver has at most two distinct pathways between any two vertices if and only if the unipotent invariant subalgebra is generated by the corresponding Cartan subalgebra in the filtered quiver representation.  
In this paper, we extend this result to all finite quivers and we refer to Definition~\ref{def:locally-k-Kronecker-quiver} for the definition of local $k$-generalized Kronecker quiver.

\begin{theorem}\label{theorem:finite-dim}
Let $Q$ be a finite quiver with at least one loop or a local $k$-generalized path Kronecker quiver, with $k\geq 2$, at one of the vertices, and let $\beta = (n,n,\ldots,n)\in \mathbb{Z}_{\geq 0}^{Q_0}$ be a dimension vector. For $W\in Rep(Q,\beta)$, let $F^{\bullet}W $ be a complete filtration of vector spaces at each vertex and  
let $\mathbb{U}=U^{\oplus Q_0}$ be the product of maximal unipotent subgroups. 
Then 
\begin{enumerate}
\item\label{item:conclusion-one} the vector subspace $F^{\bullet}Rep(Q,\beta)$ contains a point $p$ such that $\Stab_{\mathbb{U}}(p)=\{ \I \}$, 
\item\label{item:conclusion-two} $\dim_{\mathbb{C}} F^{\bullet}Rep(Q,\beta)/\!\!/ \mathbb{U} = \dim_{\mathbb{C}} F^{\bullet}Rep(Q,\beta) - \dim_{\mathbb{C}} \mathbb{U}$. 
\end{enumerate} 
\end{theorem}

\begin{theorem}\label{theorem:fin-dim-all-quivers}
Let $Q$ be a finite quiver and let $\beta =  (n,n,\ldots,n)\in \mathbb{Z}_{\geq 0}^{Q_0}$. For $W\in Rep(Q,\beta)$, let $F^{\bullet}W $ be a complete filtration of vector spaces at each vertex and  
let $\mathbb{U}=U^{\oplus Q_0}$. 
Then  
$\dim_\mathbb{C} \spec(\mathbb{C}[F^{\bullet}Rep(Q,\beta)]^{\mathbb{U}})$ is finite. 
\end{theorem}

If $Q$ has no arrows, then there is nothing to show for Theorem~\ref{theorem:fin-dim-all-quivers}. Thus, assuming $Q$ has one or more arrows,   
Theorem~\ref{theorem:fin-dim-all-quivers} follows from Theorem~\ref{theorem:finite-dim} and Theorem 1.1 in \cite{Im-filtered-semi-invariants} as these two results exhaustively cover cases for all quivers with at least one arrow. 
Theorem~\ref{theorem:fin-dim-all-quivers} also shows that 
$\Proj(\oplus_{i\geq 0} \mathbb{C}[\mu_{\mathbb{P}_{\beta}}^{-1}(0)]^{\mathbb{P}_{\beta},\chi^i})$ is a finite dimensional, projective scheme, where $\mu_{\mathbb{P}_{\beta}}:T^*(F^{\bullet}Rep(Q,\beta))\rightarrow \lie(\mathbb{P}_{\beta})^*$ is the moment map for the filtered representation space. 

Secondly, we restrict to a framed $1$-Jordan quiver and prove the existence between the $B$-Hamiltonian reduction of $\mathfrak{b}\times V$ and the isospectral Hilbert scheme 
(see Theorem~\ref{thm:birationality-Y/B-IsoHilb}).   
We also discuss filtered representations in the setting of quiver-graded Steinberg varieties (Section~\ref{section:quiver-Steinberg-var}), quantum Hamiltonian reduction (Section~\ref{section:quantum-Hamiltonian-reduction}), and quantization deformation (Section~\ref{section:deformation-quantization}), providing motivation and open problems in each section. We review symplectic reflection algebras and their spherical subalgebras in Section~\ref{subsection:rational-Cherednik-algebras}.  We recall the construction of spherical subalgebras of Cherednik algebras,which are related to quantized Hamiltonian reduction in the classical setting. 
Using this, we give a conjecture which relates rational Cherednik algebras of type $A$ to the quantized Hamiltonian reduction for the Borel setting.  We end with future directions in Section~\ref{section:applications-future-work}.

\section{Background}\label{section:background}  
 
Although thorough discussions on the algebraic and geometric aspects of quiver representations are given in   
\cite{Brion-rep-of-quivers},  
\cite{Crawley-Boevey-rep-quivers},  
\cite{Ginzburg-Nakajima-quivers},   
\cite{Im-doctoral-thesis}, and  
\cite{MR1315461},  
we will give a few foundational definitions in this section for completeness of this paper.

A quiver $Q=(Q_0, Q_1)$ is a directed graph with a set $Q_0=\{1,2,\ldots, p\}$ of vertices 
and a set $Q_1=\{ a_1,a_2,\ldots, a_q\}$ of arrows, 
which come equipped with two functions: for each arrow  $\stackrel{i}{\bullet}
            \stackrel{a}{\longrightarrow} \stackrel{j}{\bullet}$ from vertex $i$ to vertex $j$, 
            $t:Q_1\rightarrow Q_0$ maps $t(a)=ta=i$ and $h:Q_1\rightarrow Q_0$ maps $h(a)=ha=j$.  
		We will call $t(a)$ the tail of arrow $a$ and $h(a)$ the head of arrow $a$. We say a quiver $Q=(Q_0,Q_1)$ is nontrivial if $|Q_0|\geq 1$,  finite if $|Q_0|<\infty$ and $|Q_1|< \infty$, and connected if the underlying graph is connected.  
Although infinite quivers play important roles (cf. \cite{rep-theory-infinite-quiver}, \cite{Zelikson-Shmuel-AR-quivers}, \cite{Se-jin-Oh}),  
we will assume our quiver is nontrivial, finite, and connected.

We say a vertex in $Q_0$ is a  sink ($+$-admissible) if it is not the head of some arrow of the quiver and the vertex is a source ($-$-admissible) if it is not the tail of some arrow of the quiver.  
A nontrivial path in $Q$ is a sequence $p=a_k\cdots a_2 a_1$ ($k\geq 1$) of arrows which satisfies $t(a_{i+1})=h(a_i)$ for all $1\leq i\leq k-1$; the path $p$ begins at the tail of $a_1$ and ends at the head of $a_k$ and we will write $h(p)=h(a_k)$ and $t(p)=t(a_1)$. The length $l(p)$ of a path $p$ is the number of arrows in the path. If  $p= a_k\cdots a_2 a_1$ is a nontrivial path, then $l(p)= k$, or else, $l(p)=0$.  

We associate a path $e_i$ to each vertex $i$ called the trivial (empty) path whose head and tail are at $i$. The length of an empty path is $0$.  
 If the tail of a nontrivial path equals the head of the path, then the path is said to be a cycle, and we say a quiver is acyclic if it has no cycles. If the nontrivial path is actually a single arrow whose tail equals its head, then the arrow is said to be a loop. 

A dimension vector $\beta$ for $Q$ is an element of $\mathbb{Z}_{\geq 0}^{Q_0}$. A representation $W$ of a quiver $Q$ assigns a vector space $W(i) = W_i$ to each vertex $i\in Q_0$ and a linear map $W(a):W(ta)\rightarrow W(ha)$ to each arrow $a\in Q_1$. A representation $W=(W(i)_{i\in Q_0}, W(a)_{a\in Q_1})$ of $Q$ is finite dimensional if each vector space $W(i)$ is finite dimensional over $\mathbb{C}$. A subrepresentation of a representation $W$ of $Q$ is a subspace $V\subseteq W$ which is invariant under all operators, i.e., $W(a)(V(ta))\subseteq V(ha)$, where $a\in Q_1$.

Now, let $W$ be a representation of $Q$ of dimension vector $\beta\in \mathbb{Z}_{\geq 0}^{Q_0}$. Upon fixing a basis for each finite-dimensional vector space $W(i)$, each $W(i)$ is identified with $\mathbb{C}^{\beta_{i}}$ and each linear map $\xymatrix{ \stackrel{W(ta)}{\bullet} \ar[r]^{W(a)} & \stackrel{W(ha)}{\bullet}}$ may be identified with a $\beta_{ha}\times \beta_{ta}$ matrix. We will thus define the quiver representation space as 
\[ 
Rep(Q,\beta) := \displaystyle{\bigoplus_{a\in Q_1 }\Hom_{\mathbb{C}}(\mathbb{C}^{\beta_{ta}}, \mathbb{C}^{\beta_{ha} })}. 
\]
Points in $Rep(Q,\beta)$ parameterize finite-dimensional representations of $Q$ of dimension vector $\beta$, and classically, there is a natural 
$\displaystyle{\mathbb{G}_{\beta}=\prod_{i\in Q_0} GL_{\beta_i}(\mathbb{C})}$-action on $Rep(Q,\beta)$  
as a change-of-basis; that is, given 
$(g_{\beta_i})_{i\in Q_0}\in \mathbb{G}_{\beta}$ and $W\in Rep(Q,\beta)$, we have 
\[ (g_{\beta_i})_{i\in Q_0}.(W(a))_{a\in Q_1} = (g_{\beta_{ha}}W(a)g_{\beta_{ta}}^{-1} )_{a\in Q_1}. 
\]  
Whenever the composition $pq$ of paths is defined,  we set $W(pq):=W(p)W(q)$, i.e., the representation of a composition of paths is the product of representations of the paths.

We will now discuss the $B$-adjoint action on $\mathfrak{b}$ further.  Let $F^{\bullet}:\{ 0\}\subseteq \mathbb{C}^1\subseteq \mathbb{C}^2 \subseteq \ldots \subseteq \mathbb{C}^n$ be the complete standard filtration of vector spaces in $\mathbb{C}^n$.  Then $\mathfrak{b}$ can be identified with the subspace of linear maps $\mathbb{C}^n\stackrel{f}{\rightarrow} \mathbb{C}^n$ such that $f|_{\mathbb{C}^k}:\mathbb{C}^k\rightarrow \mathbb{C}^k$ for all $k$. Since one must preserve the filtration of vector spaces while changing the basis, we have the $B$-action on the domain and the codomain.   So points of $\mathfrak{b}/B$ correspond to equivalence classes of linear maps preserving the complete standard filtration of vector spaces, where two maps $f$ and $g$ are equivalent if there exists a change-of-basis that will take $f$ to $g$.

We begin by giving the construction of filtered quiver representations in the general setting. 
Let $Q=(Q_0,Q_1)$ be a quiver and let $\beta=(\beta_1,\ldots, \beta_{Q_0})\in \mathbb{Z}_{\geq 0}^{Q_0}$ be a dimension vector.  
Let 
$F^{\bullet}:0\subseteq \mathbb{C}^{\gamma_1}\subseteq \mathbb{C}^{\gamma_2}\subseteq \ldots \subseteq \mathbb{C}^{\beta}$ 
be a filtration of vector spaces such that we have the filtration 
$F_i^{\bullet}:0\subseteq \mathbb{C}^{\gamma_1^i}\subseteq \mathbb{C}^{\gamma_2^i}\subseteq \ldots \subseteq \mathbb{C}^{\beta_i}$ of vector spaces at vertex $i$. 
Let $Rep(Q,\beta)$ be the quiver representation in the classical sense (without the filtration of vector spaces imposed). 
Then $F^{\bullet}Rep(Q,\beta)$ is a subspace of $Rep(Q,\beta)$ whose linear maps preserve the filtration of vector spaces at every level.  
Let $P_i\subseteq GL_{\beta_i}(\mathbb{C})$ be the maximal parabolic group preserving the filtration of vector spaces at vertex $i$. 
Then the product $\mathbb{P}_{\beta}:=\prod_{i\in Q_0}P_i$ of parabolic groups acts on 
$F^{\bullet}Rep(Q,\beta)$ as a change-of-basis.

Now, given a parabolic Lie algebra $\mathfrak{p}=\lie(P)$, a parabolic matrix described above corresponding to a filtration of vector spaces with respect to the standard basis is block upper triangular. A general parabolic matrix has indeterminates along its block diagonal and upper triangular portion of the matrix and zero below the diagonal blocks. We refer to \cite{MR2772068} for a discussion on quiver Grassmannians and quiver flag varieties, which are related to filtered quiver representations but they are not the same. In the next two sections, we describe quiver flag varieties and filtered quiver representations.

\subsection{Quiver flag varieties}\label{subsection:quiver-flag-variety}

In this section, we will discuss the notion of quiver flag varieties, which also appear in the literature as quiver flag manifolds. 
First fixing a dimension vector $\beta=(\beta_1,\ldots, \beta_{Q_0})$, let $\gamma^{(1)}, \ldots, \gamma^{(l)}\in \mathbb{Z}_{\geq 0}^{Q_0}$
be dimension vectors with coordinates 
$\gamma^{(k)}=(\gamma_1^{(k)},\ldots, \gamma_{Q_0}^{(k)})$ 
satisfying $\sum_{k=1}^l \gamma_i^{(k)}\leq \beta_i$ for each $i\in Q_0$. 
Let $\gamma^{(l+1)}\in \mathbb{Z}_{\geq 0}^{Q_0}$ such that  
$\sum_{k=1}^{l+1} \gamma_i^{(k)} = \beta_i$ for each $i\in Q_0$.  
Define $Fl_{\gamma^{\bullet}}(\beta) := \prod_{i\in Q_0} Fl_{\gamma_i^{\bullet}}(\beta_i)$  
to be the product of flag varieties,  
where each $Fl_{\gamma_i^{\bullet}}(\beta_i)$
is the usual flag variety parametrizing flags of subspaces 
\begin{equation}
0 \subseteq U_i^{(1)}\subseteq U_i^{(2)} \subseteq \ldots \subseteq U_i^{(l)} \subseteq W(i) \mbox{ with } \dim U_i^{(k)} = \sum_{u=1}^k \gamma_i^{(u)}. 
\end{equation}
We define the universal quiver flag to be:  
\begin{equation}\label{equation:universal-quiver-flag}
  \begin{aligned}  
  Fl_{\gamma^{\bullet}}^Q &(\beta) :=    \{ (U^{(1)},\ldots, U^{(l)}, W) \in Fl_{\gamma^{\bullet}}(\beta)\times Rep(Q,\beta): 
    0 \subseteq U^{(1)} \subseteq U^{(2)} \subseteq \ldots \subseteq U^{(l)} \subseteq W \mbox{ is } \\
    &\mbox{ a chain of subrepresentations of }  W   
    \mbox{ and } W(a)U_i^{(k)}\subseteq U_j^{(k)}\:\: \forall \: a:i\rightarrow j   \mbox{ and }  
    \forall \: 1\leq k\leq l    \}.   \\ 
    \end{aligned}
 \end{equation}

Consider the two $\mathbb{G}_{\beta}$-equivariant projections $Fl_{\gamma^{\bullet}}^Q(\beta)\stackrel{p_1}{\longrightarrow} Fl_{\gamma^{\bullet}}(\beta)$  and  $Fl_{\gamma^{\bullet}}^Q(\beta)\stackrel{p_2}{\longrightarrow} Rep(Q,\beta)$. 
First, let us view $U_i^{(k)}$ as subspaces in $W(i)$ for each $i\in Q_0$. 
The fiber of $p_1$ over the tuple  $(U^{(1)},\ldots, U^{(l)})$ is  again a homogeneous vector bundle isomorphic to 
\begin{equation}\label{eq:quiver-flag-vector-bundle}   
\bigoplus_{a\in Q_1}  \left(  \bigoplus_{k=1}^{l}\Hom\left(U_{ta}^{(k)}/U_{ta}^{(k-1)}, U_{ha}^{(k)} \right)   \oplus \Hom\left(W(ta)/U_{ta}^{(l)},W(ha)\right)  \right) 
\end{equation}
where $U_{ta}^{(k)}/U_{ta}^{(k-1)}:=\{v\in U_{ta}^{(k)}:v\perp u \:\:\forall u\in U_{ta}^{(k-1)} \}$, 
$W(ta)/U_{ta}^{(l)} := \{ v\in W(ta): v\perp u \:\:\forall u\in U_{ta}^{(l)} \}$, 
and $U_{ta}^{(0)}:=0$. So $p_1$ is flat. 
On the other hand, 
$p_2^{-1}(W)$ parameterizes all flags 
$0 \subseteq U^{(1)} \subseteq U^{(2)} \subseteq \ldots \subseteq U^{(l)} \subseteq W$ of subspaces with prescribed dimension vectors $\gamma^{(k)}$ with each $U^{(k)}$ being a subrepresentation of $U^{(k+1)}$ and $U^{(l)}$ being a subrepresentation of $W$. So $p_2$ is projective.  
The fiber $p_2^{-1}(W)=Fl_{\gamma^{\bullet}}(W)$ of $p_2$ over $W$ is called the {\em quiver flag variety}.

\begin{example}\label{example:quiver-flag-A1}
Consider the $A_1$-Dynkin quiver (this is the quiver whose underlying graph is an $A_1$-Dynkin graph). Let $\beta=n$, $l=n-1$, 
 and $\gamma^{(i)}=i$ for $1\leq i\leq l$. Then the quiver flag variety $Fl_{\gamma^{\bullet}}(\mathbb{C}^n)$ is isomorphic to the complete flag variety of $\mathbb{C}^n$. 
\end{example}  
 
\begin{example}\label{example:quiver-flag-generalizing-grassmannian}
For $Q$ any quiver and $l=1$, we obtain a quiver Grassmannian. 
\end{example}

\begin{example}\label{example:quiver-grassmannian-A1}
Consider the $A_1$-quiver and let $\beta=n$. Then $Rep(Q,\beta)\cong \mathbb{C}^n$, an $n$-dimensional vector space. Now let the dimension vector $\gamma$ be $m \leq n$. We obtain $Fl_{\gamma^{\bullet}}^Q(\beta) = Gr_{m}^{Q}(n) \cong Gr_m(n)\times \mathbb{C}^n$. This means the quiver Grassmannian $Gr_m(\mathbb{C}^n) = Gr_m(n)$ coincides with the classical Grassmannian.  
\end{example}

Now we will investigate the fibers of $p_1$ in Section~\ref{subsection:filtered-quiver-varieties}.

\subsection{Filtered quiver representations}\label{subsection:filtered-quiver-varieties}

Filtered quiver representations are precisely the fibers of $p_1$ over a flag of vector spaces in Section~\ref{subsection:quiver-flag-variety}. 
A more straight-forward construction is as follows:  let $Q$ be any quiver. 
Fix a set of dimension vectors $\gamma^{(1)},\ldots, \gamma^{(l)}, \beta \in \mathbb{Z}_{\geq 0}^{Q_0}$ such that $\sum_{k=1}^{l} \gamma_{i}^{(k)} \leq \beta_{i}$ for each $i\in Q_0$.  Let $F^{\bullet}: 0 \subseteq U^{(1)} \subseteq U^{(2)} \subseteq \ldots \subseteq U^{(l)} \subseteq W$ be a flag of subspaces such that $\dim U_i^{(k)} = \sum_{u=1}^k \gamma_i^{(u)}$ for each $i\in Q_0$.

\begin{definition}\label{definition:filtered-quiver-representation}   
The filtered quiver representation is a vector space (an affine variety) defined as 
\[ 
F^{\bullet}Rep(Q,\beta):= \{ W\in Rep(Q,\beta):
W(U_{ta}^{(k)})\subseteq U_{ha}^{(k)}\:\: \forall \: 1\leq k\leq l, \:\:\forall \:a\in Q_1
\}. 
\] 
If $P_i\subseteq GL_{\beta_i}(\mathbb{C})$ is a parabolic subgroup acting as a change-of-basis while preserving the filtration of vector spaces at vertex $i\in Q_0$, 
then $\mathbb{P}_{\beta}:= \prod_{i\in Q_0}P_i$ naturally acts on $F^{\bullet}Rep(Q,\beta)$.   
\end{definition}  

A filtered quiver representation is a representation space where the filtration is a structure on a representation, not on the quiver itself. We write $\mathbb{U}:= \mathbb{U}_{\beta}$, the unipotent radical of $\mathbb{P}_{\beta}$.  

We define the map from the set in~\eqref{equation:universal-quiver-flag} to the projection onto its second component  
as quiver Grothendieck-Springer resolution. Generalized Grothendieck-Springer resolution is known in the literature as follows:  
$\widetilde{G}_H := G\times_{H,\Ad}H = \ind_H^G(H)$, 
where $H$ is a closed subgroup of a linear algebraic group $G$ over a field $k$ and $H$ acts on $G\times H$ via $h.(g,x) = (gh^{-1},\ad_h(x))$. This implies that $\widetilde{G}_H$ has a $G$-action by $g'.([g,h])=[g'g,h]$. There is an embedding $\pr\times a: \widetilde{G}_H\hookrightarrow G/H\times G$, where  
$[g,h]\mapsto ([g], ghg^{-1})$. The image of $\pr\times a$ consists of $\im(\pr\times a)=\{ ([g],x): \forall g\in G \mbox{ of } [g], x\in H_g := gHg^{-1}\}$. 
The two projections 
\[
\xymatrix@-1pc{ 
& & \ar[lldd]_{\pr} \widetilde{G}_H \ar[rrdd]^a & & \\ 
& & & &  \\ 
G/H& & & & G \\ 
}
\]  
show that for all $[g]\in G/H$, the isomorphism $\pr^{-1}([g])\stackrel{\cong}{\longrightarrow}H_g$ is given through the map $a$ 
(cf. \cite{Kazhdan-Varshavsky}, Section 1.4). 
On the other hand, generalized Springer resolution is known as 
$T^*(G/P)\stackrel{f_{\mathfrak{p}}}{\rightarrow} \overline{\mathcal{O}_{\mathfrak{p}}}$, where $\mathcal{O}_p$ is the Richardson orbit associated to $\mathfrak{p}=\lie(P)$ and $P$ is a parabolic subgroup of a semisimple complex algebraic group $G$; 
 Richardson orbit for $P$ is an open, dense orbit in the nilradical of $\mathfrak{p}$ (cf. \cite{MR732546}).   
 Thus, in this paper, we will often use the terminology quiver (Grothendieck)-Springer resolutions.  
We also refer the reader to Section~\ref{section:quiver-Steinberg-var} for a further discussion of this topic.

If one assumes $Q$ to be the $1$-Jordan quiver, $\beta=n$, and assume a complete filtration of vector spaces on the representation space of $Q$, then we obtain $F^{\bullet}Rep(Q,n)\cong \mathfrak{b}$ under the $U_n$ or $B$ action, where $B$ is the Borel subgroup of $GL_n(\mathbb{C})$, and we have the identification: 
\[ 
F^{\bullet}Rep(Q,n)/\!\!/_{\chi} B \cong \widetilde{\mathfrak{g}}/\!\!/_{\chi}G, \mbox{ where } \widetilde{\mathfrak{g}} = \{ (x,\mathfrak{b})\in \mathfrak{g}\times G/B: x\in \mathfrak{b} \}.  
\]

In this paper, we fix the components of the quiver dimension vector to be a nonnegative integer $n$: $\beta = (n,n,\ldots, n)$ 
and we will also fix the standard basis $\mathbf{e}_1,\ldots, \mathbf{e}_n$ for the filtered representation space all throughout the paper. 
We will assume $F^{\bullet}$ to mean the complete standard filtration of vector spaces at each vertex. 
So without loss of generality, we will make the identifications: 
\begin{equation} 
F^{\bullet}Rep(Q,\beta) \cong \mathfrak{b}^{\oplus Q_1}, \hspace{4mm} 
\mathbb{U}_{\beta} \cong U^{\oplus Q_0}, \hspace{4mm} \mathbb{P}_{\beta} \cong B^{\oplus Q_0},  
\end{equation}
where $\mathfrak{b}$ is the set of upper triangular matrices in $\mathfrak{gl}_n$ and $\mathfrak{b}=\lie(B)$. 
Note when rewriting the direct sum of vector bundles in (\ref{eq:quiver-flag-vector-bundle}) as a product of matrices (one matrix for each arrow $a\in Q_1$), each matrix has the form of an upper triangular matrix.

Next, we give definitions which are critical in the proof of our main result. 
\begin{definition}\label{defn:superdiag-subdiag}
The main diagonal of an $n\times n$ matrix is called {\em level $0$}.   
{\em Level $k$-diagonal} or {\em $k$-superdiagonal} of an $n\times n$ matrix are those entries that are $k$ entries to the right of the main diagonal entries, and 
{\em level $(-k)$-diagonal} or {\em $k$-subdiagonal} of an $n\times n$ matrix are those entries  that are $k$ entries to the left of the main diagonal entries.  
\end{definition}

Note that level $1$-diagonal or $1$-superdiagonal are the matrix entries immediately above the diagonal entries, 
while level $(-1)$-diagonal or $1$-subdiagonal consists of matrix entries immediately below the main diagonal entries. 


The distance between two vertices in a graph is the number of edges in a shortest path connecting them. The notion of the distance is also known as graph geodesic or geodesic distance. If there is no path connecting two vertices (for example, the vertices belong to two different connected components), then we define the distance between them as infinite. 
In the case of a directed graph, the distance $d(i,j)$ between vertices $i$ and $j$ is the length of a shortest path from $i$ to $j$ consisting of arrows or arcs. We note that $d(i,j)$ need not equal $d(j,i)$ and it is possible for only one of the two to be defined. 

\begin{definition}\label{def:uni-matrix-locally-near-vertex}
We say a unipotent matrix representation at vertex $j$ is {\em locally near vertex $i$} if $d(i,j)\leq 1$ or $d(j,i)\leq 1$.  
\end{definition}

\begin{definition}\label{def:locally-k-Kronecker-quiver}   
A {\em $k$-generalized Kronecker quiver} consists of two vertices $i$ and $j$ and $k$ arrows: $a_1,\ldots, a_m:i\rightarrow j$ 
and  $a_{m+1},\ldots, a_k:j\rightarrow i$.  
A {\em $k$-generalized path Kronecker quiver} consists of two vertices $i$ and $j$ and nontrivial paths 
\[ 
\begin{aligned}
a_1^{(1)}\cdots a_{p_1}^{(1)}, \: a_1^{(2)}\cdots a_{p_2}^{(2)},    \ldots, a_1^{(m)}\cdots a_{p_m}^{(m)}&: i\rightarrow j \mbox{ and }    \\ 
b_1^{(1)}\cdots b_{q_1}^{(1)}, \:  b_1^{(2)}\cdots b_{q_2}^{(2)},    \ldots, b_1^{(n)}\cdots b_{q_n}^{(n)}&: j\rightarrow i     \\ 
\end{aligned}
\] 
such that $m + n = k$ and the paths  
$a_1^{(\iota)}\cdots a_{p_{\iota}}^{(\iota)}$  
and $b_1^{(\gamma)}\cdots b_{q_{\gamma}}^{(\gamma)}$  
do not contain a cycle. 
 
A {\em local $k$-generalized Kronecker quiver} at vertex $i$ is a subquiver consisting of two vertices, one vertex $j\not=i$ and vertex $i$, and $k$ arrows $a_1,\ldots, a_m : i\rightarrow j$ and $a_{m+1},\ldots, a_k:j\rightarrow i$.  
A {\em local $k$-generalized path Kronecker quiver} is a subquiver consisting of two vertices $i$ and $j$ and nontrivial paths 
\[ 
\begin{aligned}
a_1^{(1)}\cdots a_{p_1}^{(1)}, \: a_1^{(2)}\cdots a_{p_2}^{(2)},    \ldots, a_1^{(m)}\cdots a_{p_m}^{(m)}&: i\rightarrow j \mbox{ and }    \\ 
b_1^{(1)}\cdots b_{q_1}^{(1)},\:  b_1^{(2)}\cdots b_{q_2}^{(2)},    \ldots, b_1^{(n)}\cdots b_{q_n}^{(n)}&: j\rightarrow i     \\ 
\end{aligned}
\] 
such that $m + n = k$ and the paths  
$a_1^{(\iota)}\cdots a_{p_{\iota}}^{(\iota)}$  
and $b_1^{(\gamma)}\cdots b_{q_{\gamma}}^{(\gamma)}$  
do not contain a cycle. 
\end{definition}

In literature, a star-shaped quiver (of any orientation) has $k$-legs, each of length $s_k$, with no loops on the central vertex, where the underlying graph of a leg of a quiver is an $A_{s_k}$-Dynkin diagram.  

\begin{definition}\label{def:star-shaped-quiver} 
A {\em $1$-step star-shaped quiver} (of any orientation) has $k$-legs, each of length $1$, with no loops on the central vertex. 
A {\em local $1$-step star-shaped quiver} (of any orientation) at vertex $i$ is a subquiver consisting of $k$-legs emanating to or from $i$, each of length $1$, with no loops on the central vertex.  
\end{definition}  
 



\begin{example}\label{example:1-step-star-shaped-quiver}
A $1$-step star-shaped quiver (of any orientation) with $8$ legs is: 
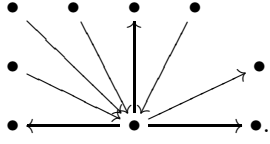
\[ 
\xymatrix@-1pc{
\bullet \ar[rrdd] & \bullet \ar[rdd] & \bullet  & \bullet  \ar[ddl] & \\ 
\bullet  \ar[rrd]& & & &\bullet  \\
\bullet   & &\ar[ll] \bullet \ar[uu]^{} \ar[rru]^{}\ar[rr]^{} & &\bullet.  \\ 
}
\] 
\end{example}

 Finally, we state Corollary 19.6 from \cite{MR1489234}, which will be applied in the proof of Theorem~\ref{theorem:finite-dim}:  
\begin{corollary}[Grosshans]\label{cor:Grosshans-dim-count} 
Let $G$ act linearly on a vector space $V$ and let $v\in V$ such that $\Stab_G(v)=\{ e\}$. 
Then 
\[ 
\dim V/\!\!/G = \dim V - \dim G. 
\] 
\end{corollary}
 

\section{Results}\label{section:results} 
Without loss of generality, we will assume $Q$ is connected and has at least one arrow. 

\begin{proposition}\label{prop:A1-Dynkin-quiver}
Let $Q$ be the $A_1$-Dynkin quiver and let $\beta= (n,n)$. 
Let $F^{\bullet}W$ be a complete filtration of vector spaces at the two vertices. Let $\mathbb{U}=U^{\oplus 2}$.   
Then $F^{\bullet}Rep(Q,\beta)/\!\!/\mathbb{U} \cong \mathbb{C}^n$. 
\end{proposition}

Also see Theorem 5.1.2 in \cite{Im-doctoral-thesis} for a generalization of Proposition~\ref{prop:A1-Dynkin-quiver} to $ADE$-Dynkin quivers.

\begin{proof}
Let $a$ be the arrow in $Q$ connecting the two vertices $1$ and $2$. 
If $1=ta$ and $2=ha$, we will call such orientation the preferred orientation. 
Define 
\begin{equation}
\epsilon(a) = 
\begin{cases} 
1 & \mbox{ if } a \mbox{ is in the preferred orientation}, \\   
0 &\mbox{ otherwise}.     \\    
\end{cases}    
\end{equation}
We identify $F^{\bullet}Rep(Q,\beta)$ with $\mathfrak{b}$. Then for $(u_1,u_2)\in \mathbb{U}$ and $x \in \mathfrak{b}$, 
suppose $(u_1,u_2).x = u_{1+\epsilon(a)} x \:  u_{2-\epsilon(a)}^{-1}$. 
At the level of functions, we have 
$(u_1, u_2).f(x)=f(u_{1+\epsilon(a)}^{-1}  x \: u_{2-\epsilon(a)})$. 
We will prove that $\mathbb{C}[\mathfrak{b}]^{\mathbb{U}} \cong \mathbb{C}[\mathfrak{t}]$, where $\mathfrak{t}$ is the Cartan subalgebra of $\mathfrak{b}$. The inclusion $\mathbb{C}[\mathfrak{t}]\subseteq \mathbb{C}[\mathfrak{b}]^{\mathbb{U}}$ is clear so we will prove the other inclusion. 

Fix a total ordering $\leq$ on pairs $(i,j)$, where $1\leq i\leq j\leq n$, by defining 
\begin{center} 
$(i,j)\leq (i',j')$ if either 
$i < i'$ or $i=i'$ and $j>j'$ 
\end{center}
and let us write $x=(x_{ij}) \in \mathfrak{b}$. 
Let $f\in \mathbb{C}[\mathfrak{b}]^{\mathbb{U}}$. 
Then for each $(i,j)$, the function $f$ can be rewritten as: 
\begin{equation}
f  = \sum_{k\geq 0}^{} x_{ij}^k f_{ij,k}, \mbox{ where }f_{ij,k} \in \mathbb{C}[\{ x_{st}: (s,t)\not=(i,j)\}]. 
\end{equation}
 Fix the least pair $(i,j)$ under the inclusion $\leq$ with $i<j$ for which there exists $k\not=0$ with $f_{ij,k}\not=0$. Continue to denote it by $(i,j)$. 
If no such $(i,j)$ exists, then $f\in\mathbb{C}[x_{ii}: 1\leq i\leq n]$ and we are done. 
Let $\widehat{u}_2$ be an $n\times n$ matrix with $1$ along the diagonal, the variable $u$ in the $(i,j)$-entry, and $0$ elsewhere. Consider $(\I,\widehat{u}_2)$.    Then 
\begin{equation}
u_{ij}.x_{st} = 
\begin{cases} 
x_{st} + x_{jt} u &\mbox{ if } s=i  \mbox{ and } \epsilon(a) = 1, \\ 
x_{st} - x_{si} u & \mbox{ if } t=j \mbox{ and } \epsilon(a) = 0, \\ 
x_{st} &\mbox{ otherwise. }  \\ 
\end{cases} 
\end{equation}
Now we rewrite $f$ as 
\[ 
f = \sum_{k\geq 0} x_{ij}^k F_k, \mbox{ where } F_k \in \mathbb{C}[\{ x_{st}: (s,t)\geq (i,j)    \} ]=: R_0. 
\] 
If $\epsilon(a)=1$, then 
\[ 
0 = u_{ij}.f -f = \sum_{k\geq 1} \sum_{1\leq l\leq k}  x_{ij}^{k-l} x_{jj}^l u^l \binom{k}{l} F_k.  
\] 
Since $\{ x_{ij}^{k-l}u^l: 1\leq l\leq k, k\geq 0 \}$ is linearly independent over $R_0$, $F_k=0$ for $k\geq 1$, contradicting the choices of $(i,j)$. Similar argument follows if $\epsilon(a)=0$. It follows that $f\in \mathbb{C}[x_{ii}]$ as claimed. 
\end{proof}

\begin{proposition}\label{prop:k-Kronecker-setting}
Let $Q$ be a $k$-generalized Kronecker quiver, where $Q$ has more than $1$ arrow. Then there is a point in the filtered representation space of $Q$ such that its stabilizer subgroup is trivial.  
\end{proposition}

\begin{proof}
Since $Q$ is a $k$-generalized Kronecker quiver, $Q$ has finite number of arrows. Label one of the vertices as $i$ and the other as $\mu$. Let us denote $a_1,\ldots, a_{p}$ as the arrows whose head is at vertex $i$ and let us denote $b_1,\ldots, b_q$ as the arrows whose tail is at vertex $i$. 
Let us write $u=(u^{(i)},u^{(\mu)})$ to be an element in the unipotent group $U^{2}:= U\times U$, where  
$u^{(i)}$ denotes the unipotent matrix representation at vertex $i$ with entries: 
\[ 
(u^{(i)})_{\iota\gamma} = 
\begin{cases} 
u_{\iota\gamma}^{(i)} &\mbox{ if }\iota < \gamma,  \\ 
1 & \mbox{ if } \iota = \gamma, \\ 
0 & \mbox{ otherwise. } \\ 
\end{cases}
\] 
For $W\in F^{\bullet}Rep(Q,\beta)$, we have the following group actions: 
\begin{enumerate}
\item $u.W(a_j) = u^{(i)}W(a_j)(u^{(\mu)})^{-1}$, 
\item  $u.W(b_j) = u^{(\mu)}W(b_j)(u^{(i)})^{-1}$.  
\end{enumerate}
Since $F^{\bullet}Rep(Q,\beta)$ is a filtered representation space, 
\[ 
\begin{aligned}  
W(a_j)_{\iota\gamma} &= 
\begin{cases} 
x_{\iota\gamma}^{(j)} &\mbox{ if }\iota\leq \gamma, \\ 
0 &\mbox{ otherwise. } \\ 
\end{cases}  \hspace{8mm}
W(b_j)_{\iota\gamma} &= 
\begin{cases} 
y_{\iota\gamma}^{(j)} &\mbox{ if }\iota\leq \gamma, \\ 
0 &\mbox{ otherwise. } \\ 
\end{cases} 
\end{aligned} 
\] 
Next, consider the map 
\begin{center}
$U^{2}\times F^{\bullet}Rep(Q,\beta)|_{\mathfrak{b}_{a_j}}\rightarrow F^{\bullet}Rep(Q,\beta)|_{\mathfrak{b}_{a_j}}$, where  $F^{\bullet}Rep(Q,\beta)|_{\mathfrak{b}_{a_j}}$ $:=$ $F^{\bullet}Rep((\{i, \mu \}, \{ a_j\} ), (n,n)) \cong \mathfrak{b}$, 
\end{center}
sending: 

\[ 
\begin{aligned} 
&(u,W)\mapsto u.W(a_j) = u^{(i)} W(a_j) (u^{(\mu)})^{-1}\\ 
&=  
\left(  
\begin{smallmatrix} 
x_{11}^{(j)} & \cdots & \displaystyle{\sum_{\stackrel{1\leq \iota \leq \gamma_1<\cdots < \gamma_m =\gamma}{1\leq m \leq \gamma}}} (-1)^{\deg(v)} u_{1\iota}^{(i)} x_{\iota \gamma_1}^{(j)} 
u_{\gamma_1\gamma_2}^{(\mu)} \cdots u_{\gamma_{m-1},\gamma}^{(\mu)} & \cdots &   \displaystyle{\sum_{
\stackrel{1\leq \iota \leq \gamma_1<\cdots < \gamma_m =n}{1 \leq m \leq n}}}  (-1)^{\deg(v)}  
u_{1\iota}^{(i)} x_{\iota \gamma_1}^{(j)} u_{\gamma_1\gamma_2}^{(\mu)} \cdots u_{\gamma_{m-1}n}^{(\mu)}  \\  
\vdots & \cdots & \vdots & \cdots & \vdots \\  
0 & \cdots & x_{\gamma\gamma}^{(j)} &  \cdots & 
 \displaystyle{\sum_{\stackrel{\gamma \leq \iota \leq \gamma_1<\cdots < \gamma_m = n }{1 \leq m \leq \gamma }}} 
(-1)^{\deg(v)} u_{\gamma\iota}^{(i)} x_{\iota \gamma_1}^{(j)} u_{\gamma_1\gamma_2}^{(\mu)} \cdots u_{\gamma_{m-1},n}^{(\mu)} \\ 
\vdots & \cdots & \vdots & \cdots & \vdots \\  
0 & \cdots & 0 & \cdots & x_{nn}^{(j)} \\  
\end{smallmatrix}\right),  
\end{aligned} 
\] 
where $v = u_{\gamma_1\gamma_2}^{(\mu)} \cdots u_{\gamma_{m-1}\gamma_m}^{(\mu)}$, 
the product of $u_{\alpha\beta}^{(\mu)}$ obtained from the group representation at vertex $\mu$. 
In the case when 
$m=1$, then we have 
$u_{\gamma\iota}^{(i)} x_{\iota \gamma_1}^{(j)} u_{\gamma_1\gamma_2}^{(\mu)}\cdots u_{\gamma_{m-1},\gamma_m}^{(\mu)}:= 
u_{\gamma\iota}^{(i)} x_{\iota \gamma_1}^{(j)}$. Thus the degree of $v$ in $u_{\gamma\iota}^{(i)} x_{\iota \gamma_1}^{(j)} u_{\gamma_1\gamma_2}^{(\mu)}\cdots u_{\gamma_{m-1},\gamma_m}^{(\mu)}$ is zero when $m=1$.

Secondly, consider the map 
$U^2 \times F^{\bullet}Rep(Q,\beta)|_{\mathfrak{b}_{b_j}}\rightarrow F^{\bullet}Rep(Q,\beta)|_{\mathfrak{b}_{b_j}}$, 
where $F^{\bullet}Rep(Q,\beta)|_{\mathfrak{b}_{b_j}}:= F^{\bullet}Rep((\{\mu, i \}, \{ b_j\} ), (n,n)) \cong \mathfrak{b}$, sending  
\[ 
\begin{aligned} 
&(u,W)\mapsto u.W(b_j) =  u^{(\mu)} W(b_j) (u^{(i)})^{-1}   \\ 
&= 
\begin{psmallmatrix} 
y_{11}^{(j)} & \cdots & \displaystyle{\sum_{\stackrel{1\leq \iota \leq \gamma_1<\cdots < \gamma_m =\gamma}{1\leq m \leq \gamma}}} (-1)^{\deg(w)} u_{1\iota}^{(\mu)} y_{\iota \gamma_1}^{(j)} u_{\gamma_1\gamma_2}^{(i)} \cdots u_{\gamma_{m-1},\gamma}^{(i)} & \cdots &   \displaystyle{\sum_{
\stackrel{1\leq \iota \leq \gamma_1<\cdots < \gamma_m =n}{1 \leq m \leq n}}} (-1)^{\deg(w)} u_{1\iota}^{(\mu)} y_{\iota \gamma_1}^{(j)} u_{\gamma_1\gamma_2}^{(i)} \cdots u_{\gamma_{m-1}n}^{(i)}  \\  
\vdots & \cdots & \vdots & \cdots & \vdots \\  
0 & \cdots & y_{\gamma\gamma}^{(j)} &  \cdots & 
 \displaystyle{\sum_{\stackrel{\gamma \leq \iota \leq \gamma_1<\cdots < \gamma_m = n }{1 \leq m \leq \gamma }}} 
(-1)^{\deg(w)} u_{\gamma\iota}^{(\mu)} y_{\iota \gamma_1}^{(j)} u_{\gamma_1\gamma_2}^{(i)} \cdots u_{\gamma_{m-1},n}^{(i)} \\ 
\vdots & \cdots & \vdots & \cdots & \vdots \\  
0 & \cdots & 0 & \cdots & y_{nn}^{(j)} \\  
\end{psmallmatrix},   \\ 
\end{aligned} 
\]  
where 
$w = u_{\gamma_1\gamma_2}^{(i)} \cdots u_{\gamma_{m-1}\gamma_m}^{(i)}$, the product of $u_{\alpha\beta}^{(i)}$ appearing in the group representation at vertex $i$. 
Similar as before, 
the degree of $w$ in 
$u_{\gamma\iota}^{(\mu)} x_{\iota \gamma_1}^{(j)} u_{\gamma_1\gamma_2}^{(i)}\cdots u_{\gamma_{m-1},\gamma_m}^{(i)}$ 
is zero if $m=1$.

Now we will restrict to the representation $W$ in $F^{\bullet}Rep(Q,\beta)$ satisfying the following conditions: for each arrow $d\in Q_1$, 
$x_{\iota\iota}^{(j')}x_{\iota+1,\iota+1}^{(j)}-x_{\iota\iota}^{(j)}x_{\iota+1,\iota+1}^{(j')}\not=0$, 
$y_{\iota\iota}^{(j')}y_{\iota+1,\iota+1}^{(j)}-y_{\iota\iota}^{(j)}y_{\iota+1,\iota+1}^{(j')}\not=0$, 
and 
$x_{\iota\iota}^{(j')}y_{\iota\iota}^{(j)} - y_{\nu\nu}^{(j)}x_{\nu\nu}^{(j')}\not=0$  
for all $j\not=j'$ 
and $1\leq \iota <\nu \leq  n$. 
  
First, suppose $Q$ has two arrows $a_j$ and $a_j'$ where $j\not=j'$. They have matrix representations 
$W( a_j )$ and $W(a_{ j' })$, where $j\not=j'$.   
So for $l\in \{ j,j'\} $,  
$u^{(i)} W(a_l) (u^{(\mu)})^{-1}$ is   
\[  
\left(  
\begin{smallmatrix} 
x_{11}^{(l)} & \cdots & \displaystyle{\sum_{\stackrel{1\leq \iota \leq \gamma_1<\cdots < \gamma_m =\gamma}{1\leq m \leq \gamma}}} (-1)^{\deg(v)} u_{1\iota}^{(i)} x_{\iota \gamma_1}^{(l)} u_{\gamma_1\gamma_2}^{(\mu)} \cdots u_{\gamma_{m-1},\gamma}^{(\mu)} & \cdots &   \displaystyle{\sum_{
\stackrel{1\leq \iota \leq \gamma_1<\cdots < \gamma_m =n}{1 \leq m \leq n}}}  (-1)^{\deg(v)}  
u_{1\iota}^{(i)} x_{\iota \gamma_1}^{(l)} u_{\gamma_1\gamma_2}^{(\mu)} \cdots u_{\gamma_{m-1}n}^{(\mu)}  \\  
\vdots & \cdots & \vdots & \cdots & \vdots \\  
0 & \cdots & x_{\gamma\gamma}^{(l)} &  \cdots & 
 \displaystyle{\sum_{\stackrel{\gamma \leq \iota \leq \gamma_1<\cdots < \gamma_m = n }{1 \leq m \leq \gamma }}} 
(-1)^{\deg(v)} u_{\gamma\iota}^{(i)} x_{\iota \gamma_1}^{(l)} u_{\gamma_1\gamma_2}^{(\mu)} \cdots u_{\gamma_{m-1},n}^{(\mu)} \\ 
\vdots & \cdots & \vdots & \cdots & \vdots \\  
0 & \cdots & 0 & \cdots & x_{nn}^{(l)} \\  
\end{smallmatrix}\right). 
\] 
The level $1$-diagonal entries of $u^{(i)} W(a_l) (u^{(\mu)})^{-1}=W(a_l)$, where $l=j,j'$, are: 
\begin{equation}\label{eq:level-1-diag-Kronecker}  
\begin{aligned}
x_{\iota, \iota+1}^{(j)} + u_{\iota,\iota+1}^{(i)} x_{\iota+1, \iota+1}^{(j)} -  x_{\iota\iota}^{(j)}  u_{\iota,\iota+1}^{(\mu)} &= x_{\iota, \iota+1}^{(j)},   \\ 
x_{\iota, \iota+1}^{(j')} + u_{\iota,\iota+1}^{(i)} x_{\iota+1, \iota+1}^{(j')} -  x_{\iota\iota}^{(j')}  u_{\iota,\iota+1}^{(\mu)} &= x_{\iota, \iota+1}^{(j')}.  \\ 
\end{aligned}
\end{equation}  
These simplify as   
\[ 
\begin{aligned}
 u_{\iota,\iota+1}^{(i)} x_{\iota+1, \iota+1}^{(j)} -  x_{\iota\iota}^{(j)}  u_{\iota,\iota+1}^{(\mu)} &= 0,   \\ 
u_{\iota,\iota+1}^{(i)} x_{\iota+1, \iota+1}^{(j')} -  x_{\iota\iota}^{(j')}  u_{\iota,\iota+1}^{(\mu)} &= 0, \\ 
\end{aligned} 
\] 
which reduce to solving a system of linear equations: 
\[ 
\left( 
\begin{matrix}
x_{\iota+1, \iota+1}^{(j)} & -  x_{\iota\iota}^{(j)}  \\ 
x_{\iota+1, \iota+1}^{(j')} &  -  x_{\iota\iota}^{(j')}  \\ 
\end{matrix} 
\right)  
\left( 
\begin{matrix} 
u_{\iota,\iota+1}^{(i)} \\   
u_{\iota,\iota+1}^{(\mu)}  \\   
\end{matrix}  
\right)  
= 
\left( 
\begin{matrix}
0 \\ 
0 \\ 
\end{matrix}  
\right).   
\]   
Since the coefficient matrix is invertible, $u_{\iota,\iota+1}^{(i)}=u_{\iota,\iota+1}^{(\mu)} =0$.  
Since $\iota$ is arbitary, we conclude that $u_{\iota,\iota+1}^{(i)}=u_{\iota,\iota+1}^{(\mu)} =0$ for all $1\leq \iota < n$. 
Now assume complete induction on $k$-superdiagonal entries; that is,  
$u_{\iota, \iota+\gamma}^{(i)} = u_{\iota,\iota+\gamma}^{(\mu)} = 0$ for all $1\leq \iota \leq n-k$ and $1\leq \gamma\leq k$. 
The following sets of equations are on the $(k+1)$-superdiagonal:   
\begin{equation}  
\begin{aligned}
  \displaystyle{\sum_{
\stackrel{\iota \leq \iota' \leq \gamma_1<\cdots < \gamma_m =\iota+k+1}{1 \leq m \leq  \iota}}}  (-1)^{\deg(v)}  
u_{\iota \iota'}^{(i)} x_{\iota' \gamma_1}^{(j)} u_{\gamma_1\gamma_2}^{(\mu)} \cdots u_{\gamma_{m-1}, \iota+k+1 }^{(\mu)}  
&= x_{\iota, \iota+k+1 }^{(j)}, \\ 
 \displaystyle{\sum_{
\stackrel{\iota \leq \iota' \leq \gamma_1<\cdots < \gamma_m =\iota+k+1}{1 \leq m \leq  \iota}}}  (-1)^{\deg(v)}  
u_{\iota \iota'}^{(i)} x_{\iota' \gamma_1}^{(j')} u_{\gamma_1\gamma_2}^{(\mu)} \cdots u_{\gamma_{m-1}, \iota+k+1 }^{(\mu)}   
&= x_{\iota, \iota+k+1 }^{(j')}. \\   
\end{aligned}   
\end{equation}   
If $m\geq 3$,   
then the difference of the indices $\gamma'$ and $\gamma''$ for the variables  
$u_{\gamma'\gamma''}^{(i)}$ and $u_{\gamma'\gamma''}^{(\mu)}$ are strictly less than $k+1$. 
Since all such monomials vanish by complete induction, we are left with those terms where $m$ in the $\gamma$ partition is strictly less than $3$: 
\begin{equation}\label{eq:2-or-more-Kronecker-setting}
\begin{aligned}
 \displaystyle{\sum_{   \iota \leq \iota' \leq \gamma_1 = \iota + k + 1 }} 
(-1)^{\deg(v)} u_{\iota\iota'}^{(i)} x_{\iota' , \iota+ k+1}^{(j)}      
+  
 \displaystyle{\sum_{   \iota \leq \iota' \leq \gamma_1 < \gamma_2 = \iota + k + 1 }} 
(-1)^{\deg(v)} u_{\iota\iota'}^{(i)} x_{\iota' \gamma_1}^{(j)} u_{\gamma_{1},\iota+k+1}^{(\mu)}  
&= x_{\iota, \iota+k+1}^{(j)},   \\    
 \displaystyle{\sum_{   \iota \leq \iota' \leq \gamma_1 = \iota + k + 1 }}   
(-1)^{\deg(v)} u_{\iota\iota'}^{(i)} x_{\iota' , \iota+ k+1}^{(j')}   
+  
 \displaystyle{\sum_{   \iota \leq \iota' \leq \gamma_1 < \gamma_2 = \iota + k + 1 }} 
(-1)^{\deg(v)} u_{\iota\iota'}^{(i)} x_{\iota' \gamma_1}^{(j')} u_{\gamma_{1},\iota+k+1}^{(\mu)}  
&= x_{\iota, \iota+k+1}^{(j')}.    \\   
\end{aligned}   
\end{equation}
From the first sum in the first equation, two cases when $u_{\iota\iota'}^{(i)}$ does not equal zero are when $\iota=\iota'$ and when $\iota'=\iota+k+1$. From the second sum in the first equation, the only case when $u_{\gamma_1,\iota+k+1}^{(\mu)}$ does not equal zero is when $\gamma_1=\iota$ since $\iota+k+1-\gamma_1$ would be strictly greater than $k$. Putting this together, we have 
\[   
\bcancel{u_{\iota\iota}^{(i)} x_{\iota, \iota+ k+1}^{(j)}}
 + u_{\iota,\iota+k+1}^{(i)} x_{\iota+k+1, \iota+ k+1}^{(j)}  
-      
u_{\iota\iota}^{(i)} x_{\iota \iota}^{(j)} u_{\iota ,\iota+k+1}^{( \mu )}    
=\bcancel{ x_{\iota, \iota+k+1}^{(j)}}, 
\] 
or 
\begin{equation}\label{eq:2-Kron-setting-towards-inv-matrix-j}
u_{\iota,\iota+k+1}^{(i)} x_{\iota+k+1, \iota+ k+1}^{(j)}  
-      
  x_{\iota \iota}^{(j)} u_{\iota ,\iota+k+1}^{( \mu )}    
=0.   
\end{equation}  
Similarly we obtain 
\begin{equation}\label{eq:2-Kron-setting-towards-inv-matrix-jprime} 
u_{\iota,\iota+k+1}^{(i)} x_{\iota+k+1, \iota+ k+1}^{(j')}  
-       
  x_{\iota \iota}^{(j')} u_{\iota ,\iota+k+1}^{( \mu )}     
= 0   
\end{equation}
when simplifying the second equation in \eqref{eq:2-or-more-Kronecker-setting}. 
For a fixed $\iota$ and forming a system of linear equations using the two equations in \eqref{eq:2-Kron-setting-towards-inv-matrix-j} and \eqref{eq:2-Kron-setting-towards-inv-matrix-jprime}, 
we obtain: 
\[  
\left( 
\begin{matrix} 
x_{\iota+k+1, \iota+ k+1}^{(j)}   &   -x_{\iota \iota}^{(j)} \\ 
x_{\iota+k+1, \iota+ k+1}^{(j')}   &  -x_{\iota \iota}^{(j')} \\ 
\end{matrix} 
\right) 
\left( 
\begin{matrix} 
u_{\iota,\iota+k+1}^{(i)} \\ 
u_{\iota ,\iota+k+1}^{(\mu)}    \\ 
\end{matrix} 
\right) 
=  
\left( 
\begin{matrix} 
0 \\   
0 \\   
\end{matrix} 
\right) 
\]  
 Since the $2\times 2$ matrix on the left is invertible, we see that  
$u_{\iota,\iota+k+1}^{(i)} = 
u_{\iota ,\iota+k+1}^{(\mu)} = 0$. 
Since $\iota$ is arbitrary, we conclude that 
$u_{\iota,\iota+k+1}^{(i)} = 
u_{\iota ,\iota+k+1}^{(\mu)} = 0$ for all $1\leq \iota < n-k$. 

Since $Q$ is a $k$-generalized Kronecker quiver, similar arguments hold for $W(b_j)$ and $W(b_{j'})$.   
Now suppose $Q$ has two arrows, one of which is $a$ whose $ha=i$ and one of which is $b$ whose $tb=i$. We will combine the representation of these two arrows to show that the stabilizer subgroup is trivial. 
For simplicity, let us suppress the subscript and write $a$ and $b$ to denote the two arrows and let us write the entries of $W(a)$ and $W(b)$ as: 
\[ 
W(a)_{\iota\gamma} 
= 
\begin{cases}  
x_{\iota\gamma} &\mbox{ if } \iota \leq \gamma, \\  
0 &\mbox{ otherwise, } \\  
\end{cases} 
\hspace{4mm} 
\mbox{ and } 
\hspace{4mm}  
W(b)_{\iota\gamma} 
= 
\begin{cases}  
y_{\iota\gamma} &\mbox{ if } \iota \leq \gamma, \\  
0 &\mbox{ otherwise. } \\  
\end{cases} 
\] 
For $1\leq \iota<n$,
consider the $(\iota,\iota+1)$-entries of the equations:  
 \begin{equation}\label{eq:one-a-one-b-Kronecker-quiver}
u^{(i)}W(a)(u^{(\mu)})^{-1}=W(a) \hspace{4mm}  \mbox{ and } \hspace{4mm} 
u^{(\mu)}W(b)(u^{(i)})^{-1}=W(b).  
\end{equation}
They are 
\[
\begin{aligned} 
x_{\iota,\iota+1}+ u_{\iota,\iota+1}^{(i)}x_{\iota+1,\iota+1} - x_{\iota\iota}u_{\iota,\iota+1}^{(\mu)} &= x_{\iota,\iota+1}  \mbox{ and } \\ 
y_{\iota,\iota+1}+ u_{\iota,\iota+1}^{(\mu)}y_{\iota+1,\iota+1} - y_{\iota\iota}u_{\iota,\iota+1}^{(i)} &= y_{\iota,\iota+1}, \\ 
\end{aligned} 
\] 
which simplify as 
\[
\begin{aligned} 
 u_{\iota,\iota+1}^{(i)}x_{\iota+1,\iota+1} - x_{\iota\iota}u_{\iota,\iota+1}^{(\mu)} &= 0  \\ 
u_{\iota,\iota+1}^{(\mu)}y_{\iota+1,\iota+1} - y_{\iota\iota}u_{\iota,\iota+1}^{(i)} &= 0 \\ 
\end{aligned} 
\] 
or 
\[ 
\left( 
\begin{matrix} 
x_{\iota+1,\iota+1} & - x_{\iota\iota} \\  
-y_{\iota\iota} & y_{\iota+1,\iota+1} \\  
\end{matrix}
\right) 
\left(  
\begin{matrix}
u_{\iota,\iota+1}^{(i)} \\ 
u_{\iota,\iota+1}^{(\mu)} \\ 
\end{matrix}
\right) 
= 
\left( 
\begin{matrix} 
0 \\ 
0 \\ 
\end{matrix}
\right).  
\] 
Since the matrix on the left is invertible by assumption, we have that $u_{\iota,\iota+1}^{(i)} = 
u_{\iota,\iota+1}^{(\mu)} = 0$.  
Now, assume complete induction on $k$-superdiagonal, which means $u_{\iota,\iota+\gamma}^{(i)}
= u_{\iota,\iota+\gamma}^{(\mu)} =0
$ for all $1\leq \iota \leq n-k$ and $1\leq \gamma \leq k$. On the $(k+1)$-superdiagonal, we have 
\begin{equation}\label{eq:Kronecker-quiver-setting-induction-xs}
 \displaystyle{\sum_{\stackrel{\iota \leq \iota' \leq \gamma_1<\cdots < \gamma_m = \iota + k + 1 }{1 \leq m \leq \iota }}} 
(-1)^{\deg(\omega')} u_{\iota\iota'}^{(i)} x_{\iota' \gamma_1}  u_{\gamma_1\gamma_2}^{(\mu)} \cdots u_{\gamma_{m-1},\iota+k+1}^{(\mu)}  
= x_{\iota, \iota+k+1},   
\end{equation} 
\begin{equation}\label{eq:Kronecker-quiver-setting-induction-ys}
 \displaystyle{\sum_{\stackrel{\iota \leq \iota' \leq \gamma_1<\cdots < \gamma_m = \iota + k + 1 }{1 \leq m \leq \iota }}} 
(-1)^{\deg(\omega'')} u_{\iota\iota'}^{(\mu)} y_{\iota' \gamma_1}  u_{\gamma_1\gamma_2}^{(i)} \cdots u_{\gamma_{m-1},\iota+k+1}^{(i)}  
= y_{\iota, \iota+k+1},   
\end{equation} 
where $\omega' = u_{\gamma_1\gamma_2}^{(\mu)} \cdots u_{\gamma_{m-1},\iota+k+1}^{(\mu)}$ in Equation~\eqref{eq:Kronecker-quiver-setting-induction-xs},   
the product of $u_{\alpha\beta}^{(\mu)}$ obtained from the group representation at vertex $\mu$,   
and  
$\omega'' = u_{\gamma_1\gamma_2}^{(i)} \cdots u_{\gamma_{m-1},\iota+k+1}^{(i)}$ in Equation~\eqref{eq:Kronecker-quiver-setting-induction-ys},  
the product of $u_{\alpha\beta}^{(i)}$ obtained from the group representation at vertex $i$.  

First, consider Equation~\eqref{eq:Kronecker-quiver-setting-induction-xs}. 
For $m\geq 3$, the difference of the indices $\gamma'$ and $\gamma''$ for the variables $u_{\gamma'\gamma''}^{(\mu)}$ must be strictly less than $k+1$. Since all such monomials vanish by complete induction, we are left with those terms where $m$ in the $\gamma$ partition is less than $3$: 
\begin{equation}\label{eq:k-Kronecker-quiver-a-b-setting}
 \displaystyle{\sum_{   \iota \leq \iota' \leq \gamma_1 = \iota + k + 1 }} 
(-1)^{\deg(v)} u_{\iota\iota'}^{(i)} x_{\iota' , \iota+ k+1}    
+  
 \displaystyle{\sum_{   \iota \leq \iota' \leq \gamma_1 < \gamma_2 = \iota + k + 1 }} 
(-1)^{\deg(v)} u_{\iota\iota'}^{(i)} x_{\iota' \gamma_1} u_{\gamma_{1},\iota+k+1}^{(\mu)}  
= x_{\iota, \iota+k+1},  
\end{equation}
From the first sum in Equation~\eqref{eq:k-Kronecker-quiver-a-b-setting}, two cases when $u_{\iota\iota'}^{(i)}$ does not equal zero are when $\iota=\iota'$ 
and when $\iota'=\iota+k+1$. 
From the second sum, the only case when $u_{\gamma_1,\iota+k+1}^{(\mu)} \not=0$ is when $\gamma_1=\iota$ 
since then $\iota+k+1-\gamma_1>k$. Putting this together, we have   
\[   
\bcancel{u_{\iota\iota}^{(i)} x_{\iota, \iota+ k+1}} 
 + u_{\iota,\iota+k+1}^{(i)} x_{\iota+k+1, \iota+ k+1} 
-      
u_{\iota\iota}^{(i)} x_{\iota \iota}  u_{\iota ,\iota+k+1}^{( \mu )}    
=\bcancel{ x_{\iota, \iota+k+1} }, 
\] 
or 
\begin{equation}\label{eq:Kronecker-final-inductive-step-xs} 
u_{\iota,\iota+k+1}^{(i)} x_{\iota+k+1, \iota+ k+1} 
-      
  x_{\iota \iota}  u_{\iota ,\iota+k+1}^{( \mu )}    
=0.   
\end{equation}  
Using a similar argument on Equation~\eqref{eq:Kronecker-quiver-setting-induction-ys},
we obtain 
\begin{equation}\label{eq:Kronecker-final-inductive-step-ys}
u_{\iota,\iota+k+1}^{(\mu)}y_{\iota+k+1,\iota+k+1} - y_{\iota\iota}u_{\iota,\iota+k+1}^{(i)} = 0. 
\end{equation}
Put Equations~\eqref{eq:Kronecker-final-inductive-step-xs} and \eqref{eq:Kronecker-final-inductive-step-ys} 
together to obtain: 
\[ 
\left(  
\begin{matrix} 
x_{\iota+k+1,\iota+k+1} & -x_{\iota\iota} \\ 
-y_{\iota\iota} & y_{\iota+k+1,\iota+k+1} \\ 
\end{matrix} 
\right) 
\left( 
\begin{matrix}  
u_{\iota,\iota+k+1}^{(i)} \\ 
u_{\iota,\iota+k+1}^{(\mu)} \\ 
\end{matrix} 
\right) 
= 
\left( 
\begin{matrix}
0 \\ 
0 \\ 
\end{matrix} 
\right).  
\]  
Since the determinant of the matrix on the left is nonzero by assumption, we conclude that   
$u_{\iota,\iota+k+1}^{(i)} = u_{\iota,\iota+k+1}^{(\mu)} =0$ for all $1\leq \iota < n-k$.   
Thus there is a point in a $k$-generalized Kronecker quiver (where $Q$ has at least two arrows) whose stabilizer subgroup is trivial. 
\end{proof}

\begin{proposition}\label{prop:k-path-Kronecker-setting}
Let $Q$ be a $k$-generalized path Kronecker quiver, where $Q$ has more than $1$ arrow. Then there is a point in the filtered representation space of $Q$ such that its stabilizer subgroup is trivial.  
\end{proposition}

To prove Proposition~\ref{prop:k-path-Kronecker-setting}, we make the following substitutions in the proof of Proposition~\ref{prop:k-Kronecker-setting}:  
first, suppose the head of the paths $a_1^{(\iota)}\cdots a_{p_{\iota}}^{(\iota)}$ is $i$ and the tail is $\mu$ and suppose 
the head of the paths $b_1^{(\iota)}\cdots b_{q_{\iota}}^{(\iota)}$ is $\mu$ and the tail is $i$. 
Replace each arrow $a_{\iota}$ in the proof of Proposition~\ref{prop:k-Kronecker-setting} with the path $a_1^{(\iota)}\cdots a_{p_{\iota}}^{(\iota)}$
and replace each arrow $b_{\gamma}$ in the proof with the path $b_1^{(\gamma)}\cdots b_{q_{\gamma}}^{(\gamma)}$. 
Since the proof of Proposition~\ref{prop:k-path-Kronecker-setting} is similar to the proof of Proposition~\ref{prop:k-Kronecker-setting}, 
we will omit the proof of Proposition~\ref{prop:k-path-Kronecker-setting}.

\begin{lemma}\label{lemma:k-star-shaped-setting}
Let $Q$ be a star-shaped quiver, where $Q$ has at least two arrows. Then  
\[ 
\dim F^{\bullet}Rep(Q,\beta)/\!\!/\mathbb{U} = \dim F^{\bullet}Rep(Q,\beta)-\dim \mathbb{U}. 
\]   
\end{lemma}

It is trivial to show the nonexistence of a point in a filtered representation space of a star-shaped quiver   
($|Q_1| > 1$) such that it has a trivial stabilizer subgroup; this is easily shown by proving that every point has a stabilizer subgroup of dimension greater than $0$. 
However for Lemma~\ref{lemma:k-star-shaped-setting}, we will use a fact from Im's dissertation (cf. Theorem 5.3.1 in \cite{Im-doctoral-thesis}):   
only the diagonal entries of the filtered representation space of a star-shaped quiver  produce unipotent invariants.

\begin{proof}
A star-shaped quiver is a quiver with at most two distinct pathways between any two vertices. 
By Theorem 1.1 in \cite{Im-filtered-semi-invariants}, 
$\mathbb{C}[F^{\bullet}Rep(Q,\beta)]^{\mathbb{U}}\cong \mathbb{C}[\mathfrak{t}^{\oplus Q_1}]$. 
This means $\dim \mathbb{C}[F^{\bullet}Rep(Q,\beta)]^{\mathbb{U}} = \dim \mathbb{C}[\mathfrak{t}^{\oplus Q_1}] = n|Q_1|$, 
and on the other hand,  
$\dim F^{\bullet}Rep(Q,\beta)-\dim \mathbb{U} = n|Q_1|$. 
\end{proof}

We now give the proof of Theorem~\ref{theorem:finite-dim}
\begin{proof}
Assume that $Q$ has at least one arrow and by assumption, $Q$ is not the $A_2$-Jordan quiver.  
We will first prove that there is a point $v$ in $F^{\bullet}Rep(Q,\beta)$ such that the stabilizer group of $v$ is trivial. 
Let $i\in Q_0$ be a vertex of $Q$. Since $Q$ is finite, the vertex $i$ is connected to a finite number of the following types of arrows: 
\begin{enumerate} 
\item arrows $a_1,\ldots, a_p$ such that only $ha_j = i$ for all $1\leq j\leq p$, 
\item arrows $b_1, \ldots, b_q$ such that only $tb_j = i$ for all $1\leq j \leq q$,  and 
\item arrows $c_1,\ldots, c_r$ such that $hc_j = t c_j =i$ for all $1\leq j\leq r$. 
\end{enumerate} 
Let $\mu_1,\ldots, \mu_p$ be the tail of $a_1,\ldots, a_p$, respectively. 
If $\mu_j=\mu_k$ for some $a_j\not= a_k$, then we replace $\mu_1,\ldots, \mu_p$ with 
the set $\mu_1,\ldots, \widehat{\mu_k}, \ldots, \mu_p$, where the hat over $\mu_k$ denotes that it has been omitted. 
Since $p$ is finite, repeat this procedure finitely-many times so that  
$\mu_1,\ldots , \mu_{p'}$ are pairwise distinct vertices. 

Now let $\nu_1,\ldots, \nu_q$ be the head of the arrows $b_1,\ldots, b_q$, respectively. 
Similar as before, replace the set $\nu_1,\ldots, \nu_q$ of vertices with $\nu_1,\ldots, \nu_{q'}$ such that the latter forms pairwise distinct vertices of $Q$. 
Now, combine $\mu_1,\ldots, \mu_{p'}, \nu_1,\ldots, \nu_{q'}$ and replace this set with 
$\mu_1,\ldots, \mu_{p''}, \nu_1,\ldots, \nu_{q''}$  
so that no vertex is listed more than once. 
 
Fix a basis in $F^{\bullet}Rep(Q,\beta)$ such that we have the identification: $F^{\bullet}Rep(Q,\beta) \cong \mathfrak{b}^{\oplus Q_1}$.
Let us write the $Q_0$-tuple in $\mathbb{U}$ as: 
\begin{equation}
u=(u^{(1)}, \ldots, u^{(i)}, \ldots, u^{(\mu_1)},\ldots,  u^{(\mu_{p''})}, \ldots, u^{(\nu_1)},\ldots, u^{(\nu_{q''})}, \ldots, u^{(Q_0)})\in \mathbb{U} \cong U^{\oplus Q_0}, 
\end{equation}  
where $u^{(i)}$ is the $n\times n$ unipotent matrix at vertex $i$ whose entries are: 
\[ 
(u^{(i)})_{\iota\gamma} = 
\begin{cases} 
u_{\iota \gamma}^{(i)} &\mbox{ if } \iota < \gamma, \\ 
1 &\mbox{ if }\iota = \gamma, \\ 
0 &\mbox{ otherwise}. 
\end{cases} 
\]  
 
For $W\in F^{\bullet}Rep(Q,\beta)$, 
we have the following group actions: since arrows $a_j$, $b_j$, and $c_j$ are connected to vertex $i$, we have 
\begin{enumerate}
\item $u.W(a_j) = u^{(i)} W(a_j) (u^{(\mu_j)})^{-1}$, 
\item $u.W(b_j) = u^{(\nu_j)} W(b_j) (u^{(i)})^{-1}$, 
\item $u.W(c_j) = u^{(i)}W(c_j) (u^{(i)})^{-1}$. 
\end{enumerate}
Here, we used the notation that $\mu_j$ is the tail of the arrow $a_j$ and 
$\nu_j$ is the head of the arrow $b_j$ (thus, $\mu_j$'s and 
$\nu_j$'s are not necessarily pairwise distinct). 
Since $F^{\bullet}Rep(Q,\beta)$ is a filtered representation space, 
\[ 
\begin{aligned}
W(a_j)_{\iota \gamma} 
	&= 
	\begin{cases} 
	x_{\iota \gamma}^{(j)} &\mbox{ if } \iota\leq \gamma, \\ 
	0 &\mbox{ otherwise}, \\ 
	\end{cases} 		\\  
W(b_j)_{\iota \gamma}  
	&=  
	\begin{cases} 
	y_{\iota \gamma}^{(j)}  &\mbox{ if } \iota\leq \gamma, \\ 
	0 &\mbox{ otherwise}, \\ 
	\end{cases} 		\\  
W(c_j)_{\iota \gamma} 
	&=  
	\begin{cases} 
	z_{\iota \gamma}^{(j)}  &\mbox{ if } \iota\leq \gamma, \\ 
	0 &\mbox{ otherwise}. \\ 
	\end{cases} 		\\  
\end{aligned}  
\]   
Consider $\mathbb{U}\times F^{\bullet}Rep(Q,\beta)|_{\mathfrak{b}_{a_j}}\rightarrow F^{\bullet}Rep(Q,\beta)|_{\mathfrak{b}_{a_j}}$, 
where $F^{\bullet}Rep(Q,\beta)|_{\mathfrak{b}_{a_j}}$ $:=$ 
$F^{\bullet}Rep((\{i,\mu_j \}, \{ a_j\} ), (n,n)) 
\cong \mathfrak{b}$, sending  
\[ 
\begin{aligned} 
&(u,W)\mapsto u.W(a_j) = u^{(i)} W(a_j) (u^{(\mu_j)})^{-1}\\ 
&=  
\left(  
\begin{smallmatrix} 
x_{11}^{(j)} & \cdots & \displaystyle{\sum_{\stackrel{1\leq \iota \leq \gamma_1<\cdots < \gamma_m =\gamma}{1\leq m \leq \gamma}}} (-1)^{\deg(v)} u_{1\iota}^{(i)} x_{\iota \gamma_1}^{(j)} u_{\gamma_1\gamma_2}^{(\mu_j)} \cdots u_{\gamma_{m-1},\gamma}^{(\mu_j)} & \cdots &   \displaystyle{\sum_{
\stackrel{1\leq \iota \leq \gamma_1<\cdots < \gamma_m =n}{1 \leq m \leq n}}}  (-1)^{\deg(v)}  
u_{1\iota}^{(i)} x_{\iota \gamma_1}^{(j)} u_{\gamma_1\gamma_2}^{(\mu_j)} \cdots u_{\gamma_{m-1}n}^{(\mu_j)}  \\  
\vdots & \cdots & \vdots & \cdots & \vdots \\  
0 & \cdots & x_{\gamma\gamma}^{(j)} &  \cdots & 
 \displaystyle{\sum_{\stackrel{\gamma \leq \iota \leq \gamma_1<\cdots < \gamma_m = n }{1 \leq m \leq \gamma }}} 
(-1)^{\deg(v)} u_{\gamma\iota}^{(i)} x_{\iota \gamma_1}^{(j)} u_{\gamma_1\gamma_2}^{(\mu_j)} \cdots u_{\gamma_{m-1},n}^{(\mu_j)} \\ 
\vdots & \cdots & \vdots & \cdots & \vdots \\  
0 & \cdots & 0 & \cdots & x_{nn}^{(j)} \\  
\end{smallmatrix}\right),  
\end{aligned} 
\]   
where $v = u_{\gamma_1\gamma_2}^{(\mu_j)} \cdots u_{\gamma_{m-1}\gamma_m}^{(\mu_j)}$, the product of $u_{\alpha\beta}^{(\mu_j)}$ obtained from the group representation at vertex $\mu_j$. 
In the case when 
$m=1$, then we have 
$u_{\gamma\iota}^{(i)} x_{\iota \gamma_1}^{(j)} u_{\gamma_1\gamma_2}^{(\mu_j)}\cdots u_{\gamma_{m-1},\gamma_m}^{(\mu_j)}:= 
u_{\gamma\iota}^{(i)} x_{\iota \gamma_1}^{(j)}$; thus the degree of $v$ in $u_{\gamma\iota}^{(i)} x_{\iota \gamma_1}^{(j)} u_{\gamma_1\gamma_2}^{(\mu_j)}\cdots u_{\gamma_{m-1},\gamma_m}^{(\mu_j)}$ is zero when $m=1$.

Secondly, consider the map 
$\mathbb{U}\times F^{\bullet}Rep(Q,\beta)|_{\mathfrak{b}_{b_j}}\rightarrow F^{\bullet}Rep(Q,\beta)|_{\mathfrak{b}_{b_j}}$, 
where $F^{\bullet}Rep(Q,\beta)|_{\mathfrak{b}_{b_j}}:= F^{\bullet}Rep((\{\nu_j, i \}, \{ b_j\} ), (n,n)) \cong \mathfrak{b}$, sending  
\[ 
\begin{aligned} 
&(u,W)\mapsto u.W(b_j) =  u^{(\nu_j)} W(b_j) (u^{(i)})^{-1}   \\ 
&= 
\begin{psmallmatrix} 
y_{11}^{(j)} & \cdots & \displaystyle{\sum_{\stackrel{1\leq \iota \leq \gamma_1<\cdots < \gamma_m =\gamma}{1\leq m \leq \gamma}}} (-1)^{\deg(w)} u_{1\iota}^{(\nu_j)} y_{\iota \gamma_1}^{(j)} u_{\gamma_1\gamma_2}^{(i)} \cdots u_{\gamma_{m-1},\gamma}^{(i)} & \cdots &   \displaystyle{\sum_{
\stackrel{1\leq \iota \leq \gamma_1<\cdots < \gamma_m =n}{1 \leq m \leq n}}} (-1)^{\deg(w)} u_{1\iota}^{(\nu_j)} y_{\iota \gamma_1}^{(j)} u_{\gamma_1\gamma_2}^{(i)} \cdots u_{\gamma_{m-1}n}^{(i)}  \\  
\vdots & \cdots & \vdots & \cdots & \vdots \\  
0 & \cdots & y_{\gamma\gamma}^{(j)} &  \cdots & 
 \displaystyle{\sum_{\stackrel{\gamma \leq \iota \leq \gamma_1<\cdots < \gamma_m = n }{1 \leq m \leq \gamma }}} 
(-1)^{\deg(w)} u_{\gamma\iota}^{(\nu_j)} y_{\iota \gamma_1}^{(j)} u_{\gamma_1\gamma_2}^{(i)} \cdots u_{\gamma_{m-1},n}^{(i)} \\ 
\vdots & \cdots & \vdots & \cdots & \vdots \\  
0 & \cdots & 0 & \cdots & y_{nn}^{(j)} \\  
\end{psmallmatrix},   \\ 
\end{aligned} 
\]  
where 
$w = u_{\gamma_1\gamma_2}^{(i)} \cdots u_{\gamma_{m-1}\gamma_m}^{(i)}$, the product of $u_{\alpha\beta}^{(i)}$ appearing in the group representation at vertex $i$. 
Similar as before, 
the degree of $w$ in 
$u_{\gamma\iota}^{(\nu_j)} x_{\iota \gamma_1}^{(j)} u_{\gamma_1\gamma_2}^{(i)}\cdots u_{\gamma_{m-1},\gamma_m}^{(i)}$ 
is zero if $m=1$.

Finally, the map 
$\mathbb{U}\times F^{\bullet}Rep(Q,\beta)|_{\mathfrak{b}_{c_j}}\rightarrow F^{\bullet}Rep(Q,\beta)|_{\mathfrak{b}_{c_j}}$, 
where $F^{\bullet}Rep(Q,\beta)|_{\mathfrak{b}_{c_j}}:=F^{\bullet}Rep((\{  i \}, \{  c_j\} ), n)\cong \mathfrak{b}$, sends 
\[  
\begin{aligned}  
&(u,W)\mapsto u.W(c_j) = u^{(i)}W(c_j) (u^{(i)})^{-1} \\   
&=  
\begin{psmallmatrix} 
z_{11}^{(j)} & \cdots & \displaystyle{\sum_{\stackrel{1\leq \iota \leq \gamma_1<\cdots < \gamma_m =\gamma}{1\leq m \leq \gamma}}} (-1)^{\deg(\omega)} u_{1\iota}^{(i)} z_{\iota \gamma_1}^{(j)} u_{\gamma_1\gamma_2}^{(i)}\cdots u_{\gamma_{m-1},\gamma}^{(i)} & \cdots &   \displaystyle{\sum_{
\stackrel{1\leq \iota \leq \gamma_1<\cdots < \gamma_m =n}{1 \leq m \leq n}}} (-1)^{\deg(\omega)} 
u_{1\iota}^{(i)} z_{\iota \gamma_1}^{(j)} u_{\gamma_1\gamma_2}^{(i)} \cdots u_{\gamma_{m-1}n}^{(i)}  \\  
\vdots & \cdots & \vdots & \cdots & \vdots \\  
0 & \cdots & z_{\gamma\gamma}^{(j)} &  \cdots & 
 \displaystyle{\sum_{\stackrel{\gamma \leq \iota \leq \gamma_1<\cdots < \gamma_m = n }{1 \leq m \leq \gamma }}} 
(-1)^{\deg(\omega)} u_{\gamma\iota}^{(i)} z_{\iota \gamma_1}^{(j)} u_{\gamma_1\gamma_2}^{(i)} \cdots u_{\gamma_{m-1},n}^{(i)} \\ 
\vdots & \cdots & \vdots & \cdots & \vdots \\  
0 & \cdots & 0 & \cdots & z_{nn}^{(j)} \\  
\end{psmallmatrix},    \\ 
\end{aligned}   
\]   
where $\omega = u_{\gamma_1\gamma_2}^{(i)} \cdots u_{\gamma_{m-1}\gamma_m }^{(i)}$, 
the product of $u_{\alpha\beta}^{(i)}$ appearing in the group representation at vertex $i$ under the condition that the second index $\beta$ of $u_{\alpha\beta}^{(i)}$ must be greater than $\gamma_1$.

Now suppose vertex $i$ has a loop or is one of the vertices of a local $k$-generalized Kronecker quiver.   
Let $W$ be a representation in $F^{\bullet}Rep(Q,\beta)$ satisfying the following conditions: 
for each arrow $d_j\in Q_1$, $W(d_j)$ has pairwise distinct, nonzero eigenvalues, and 
$x_{\iota\iota}^{(j')}x_{\iota+1,\iota+1}^{(j)}-x_{\iota\iota}^{(j)}x_{\iota+1,\iota+1}^{(j')}\not=0$, 
$y_{\iota\iota}^{(j')}y_{\iota+1,\iota+1}^{(j)}-y_{\iota\iota}^{(j)}y_{\iota+1,\iota+1}^{(j')}\not=0$, 
and  
$x_{\iota\iota}^{(j')}y_{\iota\iota}^{(j)} - y_{\nu\nu}^{(j)}x_{\nu\nu}^{(j')}\not=0$    
for all $j\not=j'$   
and $1\leq \iota <\nu \leq  n$.

Consider one of the loops at vertex $i$, say $c_j$.  
We will analyze both sides of the equation: 
\begin{equation}\label{eq:stabilizer-subgroup-loops}
u^{(i)}W(c_j)(u^{(i)})^{-1} = W(c_j).   
\end{equation}

First consider level $0$ diagonal entries of the left hand side in \eqref{eq:stabilizer-subgroup-loops}. Since these entries are invariant under conjugation by a unipotent subgroup, 
we are done since the diagonal entries of $W(c_j)$ are also $z_{\iota\iota}^{(j)}$.   
Thus, consider $1$-superdiagonal entries, i.e., $(\iota,\iota+1)$-entries, which are of the form:   
\begin{equation}\label{eq:level1-superdiagonal} 
z_{\iota, \iota+1}^{(j)} + u_{\iota, \iota+1}^{(i)} (z_{\iota+1,\iota+1}^{(j)} - z_{\iota \iota}^{(j)}).  
\end{equation}   
Since we want to find the stabilizer subgroup of $W$, 
we set the expression in \eqref{eq:level1-superdiagonal} equal to $z_{\iota, \iota+1}^{(j)}$. 
Since  
$z_{\iota \iota}^{(j)}\not= z_{\iota+1,\iota+1}^{(j)}$, we see that $u_{\iota, \iota+1}^{(i)} =0$ for each $1\leq \iota\leq n-1$.  
Next, assume complete induction on $k$-superdiagonal:  
this means $u_{\iota, \iota+\gamma}^{(i)}=0$ for all $1\leq \iota \leq n-k$ and $1\leq \gamma\leq k$.   
On the $(k+1)$-superdiagonal,   
we have  
\begin{equation}\label{eq:loops-level-kplus1-inductive-step}  
 \displaystyle{\sum_{\stackrel{\iota \leq \iota' \leq \gamma_1<\cdots < \gamma_m = \iota + k + 1 }{1 \leq m \leq \iota }}} 
(-1)^{\deg(\omega)} u_{\iota\iota'}^{(i)} z_{\iota' \gamma_1}^{(j)} u_{\gamma_1\gamma_2}^{(i)} \cdots u_{\gamma_{m-1},\iota+k+1}^{(i)}  
= z_{\iota, \iota+k+1}^{(j)}.  
\end{equation}  
If $m\geq 3$, then the difference between the indices $\gamma'$ and $\gamma''$ for $u_{\gamma'\gamma''}^{(i)}$ must be strictly less than $k+1$; all such terms vanish by complete induction. Thus we are left with those terms when $m\leq 2$, i.e., 
\[ 
 \displaystyle{\sum_{   \iota \leq \iota' \leq \gamma_1 = \iota + k + 1 }} 
(-1)^{\deg(\omega)} u_{\iota\iota'}^{(i)} z_{\iota' , \iota+ k+1}^{(j)}      
+  
 \displaystyle{\sum_{   \iota \leq \iota' \leq \gamma_1 < \gamma_2 = \iota + k + 1 }} 
(-1)^{\deg(\omega)} u_{\iota\iota'}^{(i)} z_{\iota' \gamma_1}^{(j)} u_{\gamma_{1},\iota+k+1}^{(i)}  
= z_{\iota, \iota+k+1}^{(j)}.    
\]   
From the first sum, two cases when $u_{\iota\iota'}^{(i)}\not=0$ are when $\iota=\iota'$ and when $\iota'=\iota+k+1$.   
From the second sum, the only case when 
$u_{\gamma_{1},\iota+k+1}^{(i)} \not= 0$ is when $\gamma_1=\iota$ for   
then $\iota+k+1-\gamma_1 > k$.   
Putting this together, we have 
\[   
\bcancel{u_{\iota\iota}^{(i)} z_{\iota, \iota+ k+1}^{(j)}}
 + u_{\iota,\iota+k+1}^{(i)} z_{\iota+k+1, \iota+ k+1}^{(j)}  
-      
u_{\iota\iota}^{(i)} z_{\iota \iota}^{(j)} u_{\iota ,\iota+k+1}^{(i)}    
=\bcancel{ z_{\iota, \iota+k+1}^{(j)}}, 
\] 
or 
\[ 
u_{\iota ,\iota+k+1}^{(i)}   
( z_{\iota+k+1, \iota+ k+1}^{(j)}  
-      
z_{\iota \iota}^{(j)} ) = 0.  
\] 
Since $z_{\iota\iota}^{(j)}$ are pairwise distinct, $u_{\iota ,\iota+k+1}^{(i)}=0$ and since $\iota$ is arbitary, we are done. 

Now, using the fact that $u^{(i)}$ is the identity matrix representation at vertex $i$, we will prove that 
$u^{(\nu_j)}=\I$ for all $1\leq j\leq q''$ and   
$u^{(\mu_j)}=\I$ for all $1\leq j\leq p''$.  

So consider 
\begin{equation}\label{eq:conjugation-on-bj-first-case}
  u^{(\nu_j)} W(b_j) (u^{(i)})^{-1} = W(b_j). 
\end{equation}  
We have shown that $u^{(i)}=\I$, so Equation~\eqref{eq:conjugation-on-bj-first-case} simplifies to  
\[  u^{(\nu_j)} W(b_j)  = W(b_j).  
\]  
Since the eigenvalues of $W(b_j)$ are nonzero,   $W(b_j)$ is invertible. 
Thus, we have $u^{(\nu_j)}=\I$ for all $1\leq j\leq q''$. 

Similarly, consider 
\begin{equation}\label{eq:conjugation-on-aj-first-case} 
  u^{(i)} W(a_j) (u^{(\mu_j)})^{-1} = W(a_j).   
\end{equation}
We know that  $u^{(i)}$  is the identity matrix and $W(a_j)$ is invertible.   
Thus, Equation~\eqref{eq:conjugation-on-aj-first-case} simplifies as $u^{(\mu_j)}=\I$ for all $1\leq j\leq p''$.   
Thus, we conclude that all the unipotent matrix representations locally near vertex $i$ are identity matrices. 
We repeat similar analysis to the other $|Q_0|-1$ vertices by extending outward from vertex $i$ and using induction at those vertices $j$ such that $d(i,j)=k$ or $d(j,i)=k$ for each $k>1$. 

Now, if there are no loops at vertex $i$, i.e., $r=0$ at all the vertices of $Q$,  
then by assumption, some of the arrows connected to $i$ form a local $k$-generalized path Kronecker quiver, where $k>1$.  Without loss of generality, name the other vertex connected to these local $k$-generalized path Kronecker quiver as $\mu$. 
By Proposition~\ref{prop:k-path-Kronecker-setting} and holding all the other unipotent group representations at vertices $\nu\not= i,\mu$ fixed (since path representations connecting vertices $i$ and $\mu$ only affect the group representations at vertices $i$ and $\mu$),   
we see that there is a point in the filtered representation space of $Q$ such that its stabilizer representations at vertices $i$ and $\mu$ are trivial. 
We then repeat a similar argument as above by extending outward from vertex $i$ and $\mu$ systematically by increasing quiver geodesic distance.

We thus conclude that there is a point in the filtered representation subspace such that its unipotent stabilizer is the trivial subgroup of $\mathbb{U}$.

Finally, 
(\ref{item:conclusion-two}) follows from (\ref{item:conclusion-one})   
and Corollary~\ref{cor:Grosshans-dim-count}.  
\end{proof}

\begin{remark} 
Suppose $Q$ is connected with at least two arrows. 
In the case if $Q$ does not have a loop or a local $k$-generalized path Kronecker subquiver, then   
$Q$ must have a local $1$-step star-shaped quiver at some vertex.  
In fact, we can deduce more than this: we have a quiver with at most one pathway from one vertex to another for if $Q$ has two or more pathways from one vertex to another, then $Q$ has a local $k$-generalized path Kronecker quiver ($k\geq 2$) or $Q$ has a loop, which is a contradiction. 
By Theorem 1.1 in \cite{Im-filtered-semi-invariants}, we have an isomorphism 
$\mathbb{C}[F^{\bullet}Rep(Q,\beta)]^{\mathbb{U}_{\beta}}\cong \mathbb{C}[\mathfrak{t}^{\oplus Q_1}]$ of algebras.   
\end{remark}

\section{Special Case: $2\times 2$ $k$-Jordan filtered representations}\label{section:Jordan-quivers}

Although it is possible to write down an explicit description of the $B$-orbits of $k$-Jordan filtered representations, we will give a description of the algebra of the $U$-orbits of the $k$-Jordan when $\beta =2$.

\begin{example} 
Let $Q$ be the $k$-Jordan quiver and let $\beta=2$.   
Then the algebra $\mathbb{C}[F^{\bullet}Rep(Q,\beta)/\!\!/U]$
of unipotent invariant polynomials is isomorphic to 
\[ 
\mathbb{C}[\{c_{\iota\iota}^{(\gamma)}, c_{12}^{(\nu)}(c_{11}^{(\mu)} - c_{22}^{(\mu)}) 
- c_{12}^{(\mu)}(c_{11}^{(\nu)} - c_{22}^{(\nu)} )  :1\leq \iota \leq 2, 1\leq \gamma\leq k, 1\leq \mu <\nu\leq k\} ]. 
\]  
Let $\mathfrak{h}_2	\subseteq \mathfrak{b}_2$, the set of $2\times 2$ upper triangular matrices over the complex numbers. 
Under the $B$-action, the polynomial 
$f(W(c_1),\ldots, W(c_k) )=c_{12}^{(\nu)}(c_{11}^{(\mu)} - c_{22}^{(\mu)}) 
- c_{12}^{(\mu)}(c_{11}^{(\nu)} - c_{22}^{(\nu)} )$ is a $b_{11}b_{22}^{-1}$-semi-invariant. 
Thus, the affine quotient of the Jordan quiver is 
$F^{\bullet}Rep(Q,2)/\!\!/B = \mathfrak{b}_2^{k}/\!\!/B \cong \spec(\mathbb{C}[\mathfrak{h}_2^{\oplus k}])\cong \mathbb{C}^{2k}$
while the GIT quotient 
$F^{\bullet}Rep(Q,2)/\!\!/_{\chi_{\iota}}B   $ is a projective scheme over $\mathbb{C}^{2k}$, where $\chi_{\iota}:B\rightarrow \mathbb{C}^*$ is a group homomorphism sending $(b_{\mu\nu})\mapsto b_{\iota\iota}b_{\iota+1,\iota+1}^{-1}$.  
\end{example}

\section{The geometry of quiver Grothendieck-Springer resolutions}\label{section:geometry} 
\subsection{The Hamiltonian reduction of almost-commuting varieties for the Borel}
Consider the moment map $\mu_B:T^*(\mathfrak{b}\times \mathbb{C}^n) \cong \mathfrak{b}\times \mathfrak{b}^* \times \mathbb{C}^n \times (\mathbb{C}^n)^*
\rightarrow \mathfrak{b}^*\cong \mathfrak{g}/\mathfrak{n}$, where $(r,s,i,j)\mapsto [r,s]+\overline{ij}$ and  
$\mathfrak{n}$ is the set of strictly upper triangular matrices and $(\mathbb{C}^n)^* $ is the dual of $\mathbb{C}^n$.  
Let 
\[ s_{\iota\iota}'=
 \left[ \tr\left( \prod_{
1\leq k\leq n, k\not=\iota 
} l_k(r) \right)\right]^{-1} 
  \tr \left(\prod_{
1\leq k\leq n, k\not=\iota
} l_k(r) \:s\right),  
\] where $l_k(r) = r -r_{kk}\I$ (cf. Proposition 1.2 in \cite{Im-rss-locus-GS-resolution}).

We describe the regular semisimple locus of generalized almost-commuting variety: 
\begin{proposition}[\cite{Im-rss-locus-GS-resolution}, Proposition 1.4]\label{prop:Im-rss-locus}
The Hamiltonian reduction $\mu_B^{-1}(0)^{rss}/\!\!/B$ is reduced and isomorphic to $\mathbb{C}^{2n}\setminus \Delta_n$, where 
$\Delta_n=\{ (x_1,\ldots, x_n,0,\ldots, 0): x_{\iota} = x_{\gamma} \mbox{ for some }\iota\not=\gamma\}$. 
So $A_0=\mathbb{C}[\mu_B^{-1}(0)^{rss}/\!\!/B] \cong \mathbb{C}[r_{11},\ldots, r_{nn}, s_{11}',\ldots, s_{nn}'][(r_{\nu\nu}-r_{\gamma\gamma})^{-1}]$, 
and the Poisson bracket on the commutative algebra $A_0$ is induced from the standard symplectic structure on $\mathbb{C}^{2n}$. 
\end{proposition}

\begin{remark} 
The algebraic closure of the semisimple locus $\mu_B^{-1}(0)^{rss}/\!\!/B$ is a symplectic manifold. 
\end{remark}

\subsection{Isospectral Hilbert scheme} 
Recall that the Hilbert scheme $\Hilb^n(\mathbb{C}^2)$ of $n$ points on a complex plane is given as the set of ideals $I\subseteq \mathbb{C}[x,y]$ such that the (complex) dimension of $\mathbb{C}[x,y]/I$ is $n$. 
There is another well-known ADHM description of $\Hilb^n(\mathbb{C}^2)$ (cf. \cite{MR2210660}, \cite{MR1711344}). 
Let $Z=\{ (r,s,i)\in \mathfrak{gl}_n\times \mathfrak{gl}_n^* \times \mathbb{C}^n: [r,s]=0 \mbox{ and } r^{\alpha}s^{\beta}i \mbox{ spans }\mathbb{C}^n\}$. 
The Lie group $G=GL_n(\mathbb{C})$ acts on $Z$ by $A.(r, s, i)=(ArA^{-1}, AsA^{-1}, Ai)$. 

\begin{theorem}[\cite{MR1711344}, Theorem 2.1]
We have an isomorphism $\Hilb^{n}(\mathbb{C}^2)\cong Z/G$, which is smooth and projective of dimension $2n$. 
\end{theorem} 
The scheme $X_n$ in the reduced fiber product  
\[ 
\xymatrix@-1pc{ 
X_n  \ar[dd]_{\rho}^{} \ar[rr] & & (\mathbb{C}^2)^n \ar[dd] \\ 
& & \\ 
\Hilb^n(\mathbb{C}^2) \ar[rr]^{\pi} & & S^n(\mathbb{C}^2) \\ 
}
\] 
is known as the isospectral Hilbert scheme, where $\rho$ is flat of degree $n!$. 

\begin{theorem}[\cite{MR1839919}] 
The scheme $X_n$ is flat over $\Hilb^n(\mathbb{C}^2)$. 
\end{theorem}

The Hilbert scheme has generically $n$ distinct points on $\mathbb{C}^2$, where these points are unordered in 
$\Hilb^n(\mathbb{C}^2)/S_n$ and ordered in $X_n$.  
This leads us to relate $\mu^{-1}_B(0)^{rss}/\!\!/B$ or $\mu^{-1}_B(0)^{rss}/\!\!/_{\det}B$ with $X_n$.  

Let 
\begin{equation}\label{eq:Y-rss-description} 
Y^{rss} = \{(r,s,i)\in \mathfrak{b}\times \mathfrak{b}^*\times \mathbb{C}^n: 
[r,s]=0, r^{\alpha}s^{\beta}i \mbox{ spans }\mathbb{C}^n, \mbox{ and } r \mbox{ is regular}
\}.  
\end{equation}
Let $B$ act on $Y^{rss}$ via $b.(r,s,i)=(brb^{-1},bsb^{-1}, bi)$. 

\begin{proposition} 
We have an isomorphism $Y^{rss}/B \cong \mu_B^{-1}(0)^{rss}/\!\!/B$ of schemes.  
\end{proposition} 

\begin{proof}
Since $r\in \mathfrak{b}$ has pairwise distinct eigenvalues, $r$ is diagonalizable under the action by $B$. For a diagonal matrix $r$, $[r,s]=0$ implies $s$ is also a diagonal matrix. 
Thus, we have a map $Y^{rss}/B \rightarrow \mu_B^{-1}(0)^{rss}/\!\!/B$ given by $(r,s,i)\mapsto (r,s,i,0)$ and 
$\mu_B^{-1}(0)^{rss}/\!\!/B \rightarrow Y^{rss}/B$ given by $(r,s,i,j)\mapsto (r,s,i)$. 
\end{proof}

Let 
\begin{equation}\label{eq:Y-scheme-description} 
Y = \{(r,s,i)\in \mathfrak{b}\times \mathfrak{b}^*\times \mathbb{C}^n: 
[r,s]=0 \mbox{ and }  r^{\alpha}s^{\beta}i \mbox{ spans }\mathbb{C}^n 
\}  
\end{equation}  
and let  
\begin{equation} 
Z^{rss} = 
\{ (r,s,i)\in \mathfrak{gl}_n\times \mathfrak{gl}_n^* \times \mathbb{C}^n: [r,s]=0, r^{\alpha}s^{\beta}i \mbox{ spans }\mathbb{C}^n, \mbox{ and } r \mbox{ is regular}\}.
\end{equation} 

\begin{theorem}\label{thm:birationality-Y/B-IsoHilb} 
There is a birational map $Y/B \DashedArrow[->,densely dashed] X_n$. 
\end{theorem}   

\begin{proof}
The morphism $Y^{rss}/B\rightarrow X_n$ factors through $Z/G\times_{S^n\mathbb{C}^2}(\mathbb{C}^2)^n$ by sending $(r,s,i)B \mapsto (r,s,i)G \times (r_{11},s_{11},\ldots, r_{nn}, s_{nn})$. 
This map is well-defined since $r\in \mathfrak{b}\subseteq \mathfrak{gl}_n$, $s\in \mathfrak{b}^*$, which could be viewed as a lower triangular matrix, and the $B$-orbits are closed.  
Now, since $r$ is regular, $r$ is diagonalizable. The Lie bracket condition $[r,s]=0$ implies $s$ is diagonal if $r$ is also diagonal. 
So the morphism $Z^{rss}/G \times_{S^2\mathbb{C}^2} (\mathbb{C}^2)^n\rightarrow Y/B$ via 
$(r,s,i)G \times (r_{11},s_{11},\ldots, r_{nn}, s_{nn}) \mapsto (r,s,i)B$ is the projection onto the first triple, and we are done. 
\end{proof} 

\begin{corollary} 
The birationality of $Y/B$ and the isospectral Hilbert scheme implies the existence of an isomorphism between the vector spaces  $H^0(Y/B, K_{Y/B}^d)$ and 
$H^0(X_n, K_{X_n}^d)$, where $K_{Y/B} = \Omega_{Y/B}^{2n}$, the canonical bundle of $2n$ forms. 
\end{corollary}   

\begin{conjecture}
There is an embedding of the scheme 
$\mu_B^{-1}(0)/\!\!/_{\det}B$ into $X_n$. 
\end{conjecture}

It would be interesting to extend Proposition~\ref{prop:Im-rss-locus} to the entire Hamiltonian reduction and give a full description of $\mu_B^{-1}(0)/B \simeq T^*(\mathfrak{b}\times \mathbb{C}^n/B)$ in terms of well-known schemes. 
For the moment, we leave this as an open problem.

\subsection{Quiver-graded Steinberg varieties}\label{section:quiver-Steinberg-var}
We recall the Grothendieck-Springer resolution: $\widetilde{\mathfrak{g}}:= G\times_B \mathfrak{b} \stackrel{p}{\rightarrow} \mathfrak{g}$, where $(g,x)B\mapsto gxg^{-1}$  and consider the fiber product 
\[
\xymatrix@-1pc{ 
\ar[dd]_{\pi_1} \widetilde{\mathfrak{g}} \times_{\mathfrak{g}} \widetilde{\mathfrak{g}} \ar[rr]^{\pi_2}  & & \widetilde{\mathfrak{g}} \ar[dd]^{p}\\ 
& & \\  
\widetilde{\mathfrak{g}} \ar[rr]_{p} & & \mathfrak{g}. \\  
}  
\]  
The fiber product $\widetilde{\mathfrak{g}} \times_{\mathfrak{g}} \widetilde{\mathfrak{g}}$ is known as the Steinberg variety, also given as the triple 
\[ 
\{ (x,B,B')\in \mathfrak{g}\times G/B \times G/B: x \in \lie(B) \cap \lie(B') \} 
\] 
(cf. \cite{MR2838836}, \cite{MR2078342}, and \cite{MR1472321}). 
Steinberg varieties are important as they provide an alternative approach to the Springer correspondence, play a central role in the proof of Deligne-Langlands conjecture for Hecke algebras (proved by Kazhdan-Lusztig), and appear in a proof of the conjectured formula 
$\dim Z_G(u)=r+2 \dim (G/B)_u$ by Grothendieck, where $u\in U\subseteq G$, a connected reductive algebraic group over an algebraically closed field, $r=\rk( G)$, and $(G/B)_u=\{ B\subseteq G: u\in B\}$ (cf. \cite{MR2510045}). 
This construction easily extends to 
the fiber product  
\begin{equation}\label{eq:quiver-graded-Steinberg-var}
\underbrace{\widetilde{\mathfrak{g}}\times_{\mathfrak{g}} \cdots \times_{\mathfrak{g}} \widetilde{\mathfrak{g}}}_{k}/G \cong \mathfrak{b}^k/B.
\end{equation} 
Equation~\eqref{eq:quiver-graded-Steinberg-var} implies the isomorphism between $G$-orbits on $\widetilde{\mathfrak{g}}\times_{\mathfrak{g}} \cdots \times_{\mathfrak{g}} \widetilde{\mathfrak{g}}$ and $B$-orbits on $\mathfrak{b}^k$. Thus the $G$-equivariant $K$-theory on quiver-graded Steinberg varieties is equivalent to $B$-equivariant $K$-theory on the product of Borel subalgebras. 

Quiver-graded Steinberg varieties may be powerful tools to prove analogous formula or a similar Deligne-Langlands conjecture for other Hecke algebras.

\section{Quantum Hamiltonian reduction for filtered quiver representations}\label{section:quantum-Hamiltonian-reduction}  
In this section, we give a short introduction to quantum Hamiltonian reduction for nonreductive groups. 
Let $Q^{\dagger} = (Q_0^{\dagger}, Q_1^{\dagger})$ be a framed quiver, where $Q_0^{\dagger} = Q_0 \coprod \{ \circ \}$ and 
$Q_1^{\dagger} = Q_1 \coprod \{ \circ \stackrel{a}{\rightarrow} \stackrel{\gamma}{\bullet} \}$;  
the framed quiver has an extra vertex, called a framed vertex, 
and one extra arrow than the quiver $Q$, which points from the framed vertex to some vertex $\gamma$ in $Q$. 
Let $\beta^{\dagger}\in \mathbb{Z}_{\geq 0}^{Q_0^{\dagger}}$ be the dimension vector for the framed quiver. 
Consider the cotangent bundle $T^*(F^{\bullet}Rep(Q^{\dagger},\beta^{\dagger}))$ of the filtered representation space, 
which is a Poisson manifold with a Poisson action by the Lie group  
$\mathbb{P}_{\beta^{}} :=\prod_{i \in Q_0^{}}P_i$. The action is given by the map  
$T^*(F^{\bullet}Rep(Q^{\dagger},\beta^{\dagger})) \rightarrow \prod_{i\in Q_0^{}} \mathfrak{p}_i^*$, where 
$\mathfrak{p}_i =\lie(P_i)$ is the Lie algebra of $P_i$ and $\mathfrak{p}_i^*$ is the dual of $\mathfrak{p}_i$. 
We then have a homomorphism of Lie algebras 
$\phi_0:\prod_{i\in Q_0} \mathfrak{p}_i \rightarrow \Vect_{\Pi}(T^*(F^{\bullet}Rep(Q^{\dagger},\beta^{\dagger})))$, where $\Vect_{\Pi}(T^*(F^{\bullet}Rep(Q^{\dagger},\beta^{\dagger})))$ is the Lie algebra of vector fields on the cotangent bundle and $\Pi\in \Gamma(T^*F^{\bullet}Rep(Q^{\dagger},\beta^{\dagger}), \wedge^2 T (T^*(F^{\bullet}Rep(Q^{\dagger},\beta^{\dagger}))))$ is a Poisson bivector whose Lie bracket with itself is zero: $[\Pi,\Pi]=0$.  

A $\mathbb{P}_{\beta^{}}$-equivariant regular map $\mu_0: T^*(F^{\bullet}Rep(Q^{\dagger},\beta^{\dagger})) \rightarrow \prod_i \mathfrak{p}_i^*$ is a {\em moment map} for the $\mathbb{P}_{\beta^{}}$-action on the cotangent bundle 
if the dual 
$\mu_0^*:\prod_i \mathfrak{p}_i \rightarrow \mathbb{C}[T^*(F^{\bullet}Rep(Q^{\dagger},\beta^{\dagger}))]$ of the moment map satisfies the following: 
given a homomorphism of Lie algebras 
$v:\mathbb{C} [T^*(F^{\bullet}Rep(Q^{\dagger},\beta^{\dagger}))] \rightarrow \Vect_{\Pi}(T^*(F^{\bullet}Rep(Q^{\dagger},\beta^{\dagger})))$ preserving the Poisson structure, 
$v(\mu_0^*(\mathbf{a}))=\phi_0(\mathbf{a})$. 
The map $\mu_0$ is Poisson since the dual $\mu_0^*$ of the moment map is a homomorphism of Lie algebras.

Let $\lie(\mathbb{P}_{\beta^{}})$ be the Lie algebra of $\mathbb{P}_{\beta^{}}$ and let $A$ be an associative algebra equipped with a $\lie(\mathbb{P}_{\beta^{}})$-action (so there is a Lie algebra map $\phi:\lie(\mathbb{P}_{\beta^{}}) \rightarrow \Der(A)$). 
We have a {\em quantum moment map} for the pair $(A,\phi)$, 
which is an associative algebra homomorphism $\mu:U(\prod_{i\in Q_0^{}} \mathfrak{p}_i)\rightarrow A$ such that for any $\mathbf{a}\in \prod_{i\in Q_0^{}} \mathfrak{p}_i$ and $b\in A$, 
we have $[\mu(\mathbf{a}),b]=\phi(\mathbf{a})b$. 

Suppose now $A$ is a filtered associative algebra such that $\gr(A)$ is a Poisson algebra $A_0$ equipped with a $\lie(\mathbb{P}_{\beta^{}})$-action $\phi_0$ and a classical moment map $\mu_0$, and suppose that $\gr(\phi)=\phi_0$. 
A {\em quantization of $\mu_0$} is a quantum moment map $\mu:U(\lie(\mathbb{P}_{\beta^{}}))\rightarrow A$ such that $\gr(\mu)=\mu_0$. 
Define the space $A^{\lie(\mathbb{P}_{\beta^{}})}$ of 
$\lie(\mathbb{P}_{\beta^{}})$-invariants as 
$A^{\lie(\mathbb{P}_{\beta^{}})} := 
\{ b\in A: [\mu(\mathbf{a}), b]=0 \mbox{ for all }\mathbf{a}\in \lie(\mathbb{P}_{\beta^{}})\}$, a subalgebra of $A$. 
Let $J$ be a left ideal of $A$ generated by $\mu(\mathbf{a})$, where $\mathbf{a}\in \lie(\mathbb{P}_{\beta^{}})$.  
We define a two-sided ideal as 
$J^{\lie(\mathbb{P}_{\beta^{}})}:= J\cap A^{\lie(\mathbb{P}_{\beta^{}})} \subseteq A^{\lie(\mathbb{P}_{\beta^{}})}$.  
The associative algebra  
$$A/\!\!/\lie(\mathbb{P}_{\beta^{}}) := A^{\lie(\mathbb{P}_{\beta^{}})}/J^{\lie(\mathbb{P}_{\beta^{}})},$$ 
the {\em quantum Hamiltonian reduction of $A$ with respect to the quantum moment map $\mu$}.

\subsection{Special case}
For the rest of this section, assume $Q$ is the $k$-Jordan quiver, $Q^{\dagger}$ is the framed $k$-Jordan quiver, and 
$\beta^{\dagger}=(n,1)$.  
We impose the complete standard filtration of vector spaces at the vertex of $Q$. 
Consider the diagonal $B$-adjoint action on $\mathfrak{b}^k$. Extend this action to the filtered representation of the framed quiver and consider the $B$-action on 
$\mathfrak{b}^k \times \mathbb{C}^n$. This action is induced to the cotangent bundle 
$T^*(\mathfrak{b}^k \times \mathbb{C}^n)$.

\begin{lemma}
Since $T^*(\mathfrak{b}^k \times \mathbb{C}^n)$ is connected and symplectic,  $\ker (\mathbb{C}[T^*(\mathfrak{b}^k \times \mathbb{C}^n)]\stackrel{v}{\rightarrow} \Vect_{\Pi}(T^*(\mathfrak{b}^k \times \mathbb{C}^n)))=\mathbb{C}$. 
\end{lemma}

\begin{proof} 
Let $v:f\mapsto \{ f,-\}$, which assigns to $f$ the Hamiltonian vector field with Hamiltonian $f$. 
Since $[\Pi,\Pi]=0$,  
the Poisson bracket satisfies the Jacobi identity: $\{ \{ f,g\},h\} + \{ \{ g,h\},f \}+ \{ \{ h,f\}, g\}=0$ for polynomials $f,g,h$ on $T^*(\mathfrak{b}^k \times \mathbb{C}^n)$. Thus, the vector space of differentiable functions on the cotangent bundle $T^*(\mathfrak{b}^k \times \mathbb{C}^n)$  has the structure of a Lie algebra. Since $v$ is a Lie algebra homomorphism, $\ker v$ consists of constant functions. 
\end{proof} 

\begin{lemma} 
The first cohomology $H^1(T^*(\mathfrak{b}^k \times \mathbb{C}^n),\mathbb{C})=0$. It follows that 
$v: \mathbb{C} [T^*(\mathfrak{b}^k \times \mathbb{C}^n)] \rightarrow \Vect_{\Pi}(T^*(\mathfrak{b}^k \times \mathbb{C}^n))$ is surjective. 
\end{lemma}

\begin{proof}
The vector space $T^*(\mathfrak{b}^k \times \mathbb{C}^n)$ is isomorphic to $\mathbb{C}^{(k(n+1)+2)n}$ so it is clear that the first cohomology vanishes. 
Consider the complex:  
\begin{equation}
0 \longrightarrow 
 \mathbb{C} \stackrel{d^{-1}} {\longrightarrow}
\mathbb{C} [T^*(\mathfrak{b}^k \times \mathbb{C}^n)] \stackrel{d^0}{\longrightarrow}   
\Vect_{\Pi}(T^*(\mathfrak{b}^k \times \mathbb{C}^n)) \stackrel{d^1}{\longrightarrow}\ldots .  
\end{equation} 
Since $H^1(T^*(\mathfrak{b}^k \times \mathbb{C}^n),\mathbb{C})=\ker d^1/\im d^0 = 0$, for any $X\in \Vect_{\Pi}(T^*(\mathfrak{b}^k \times \mathbb{C}^n))$ whose $d^1(X)=0$, there is a polynomial $f$ on 
$T^*(\mathfrak{b}^k \times \mathbb{C}^n)$ such that $d^0(f) = X$. Thus, $v$ is surjective. 
\end{proof}

We end this section with a remark that the study of cohomological invariants arising from the (quantum) Hamiltonian reduction in the nonreductive setting may emit interesting geometric invariants.

\section{Towards deformation quantization}\label{section:deformation-quantization} 
Kontsevich in \cite{MR2062626} affirmed that Poisson manifolds admits a quantization. 
We give a short application of our filtered quiver representations to deformation quantization in the nonreductive setting to make it a tool to study interesting varieties arising from our construction. In particular, such tool may be a necessity when studying flat families to investigate the preimage of zero of a moment map for nonreductive groups.

Suppose $A$ is a deformation quantization of a Poisson algebra $A_0$ equipped with a $\lie(\mathbb{P}_{\beta^{}})$-action $\phi_0$ and a classical moment map $\mu_0$, and assume that $\phi=\phi_0 \mod \hbar$. A quantization of $\mu_0$ is a quantum moment map 
$\mu:U(\lie(\mathbb{P}_{\beta^{}}))\rightarrow A[\hbar^{-1}]$ satisfying $\mu(\mathbf{a})=\hbar^{-1}\mu_0(\mathbf{a})+\mathcal{O}(\hbar)$ for $\mathbf{a}\in \lie(\mathbb{P}_{\beta^{}})$. 

Let $A$ be a deformation quantization of the algebra $A_0=\mathbb{C}[T^*(F^{\bullet}Rep(Q^{\dagger},\beta^{\dagger}))]$ 
and let $\mu_0:\lie(\mathbb{P}_{\beta^{}})\rightarrow A_0$ be a classical moment map as before. 
Let $\mu$ be a quantum moment map quantizing $\mu_0$. Let $\mathcal{O}\subseteq \lie(\mathbb{P}_{\beta^{}})^*$ be a $\mathbb{P}_{\beta}$-invariant closed orbit, and $\Red(T^*(F^{\bullet}Rep(Q^{\dagger},\beta^{\dagger})), \mathbb{P}_{\beta^{}},\mathcal{O})$ be the corresponding classical reduction. 
Let $\lie(\mathbb{P}_{\beta^{}})_{\hbar}$ be the Lie algebra over $\mathbb{C}[[\hbar]]$ with Lie bracket 
$[\mathbf{a},\mathbf{b}]_{\hbar}=\hbar[\mathbf{a},\mathbf{b}]$, where 
$\mathbf{a},\mathbf{b}\in \lie(\mathbb{P}_{\beta^{}})$. Note that 
$\lie(\mathbb{P}_{\beta^{}})_{\hbar}\cong \lie(\mathbb{P}_{\beta^{}})[[\hbar]]$ as a vector space. 
Let $U(\lie(\mathbb{P}_{\beta^{}})_{\hbar})$ be the enveloping algebra of $\lie(\mathbb{P}_{\beta^{}})_{\hbar}$; it is a deformation quantization of the symmetric algebra $S(\lie(\mathbb{P}_{\beta^{}}))$, which is a completion of the Rees algebra of $U(\lie(\mathbb{P}_{\beta^{}}))$. 
We can now construct a quantum moment map $\mu_{\hbar}:U(\lie(\mathbb{P}_{\beta^{}})_{\hbar})\rightarrow A$ given by $\mu_{\hbar}(\mathbf{a})=\hbar \mu(\mathbf{a})$ for $\mathbf{a}\in \lie(\mathbb{P}_{\beta^{}})$.

Let $J_0\subseteq S(\lie(\mathbb{P}_{\beta^{}}))$  be an ideal of functions vanishing on the closed orbit $\mathcal{O}$ and let $J\subseteq U(\lie(\mathbb{P}_{\beta^{}})_{\hbar})$ be an ideal deforming $J_0$. 
In the case when $\mathcal{O}$ is a semisimple orbit or in the case of reductive setting, the ideal $J$ exists. 
We define the quantum reduction as 
$$\Red(A,\lie(\mathbb{P}_{\beta^{}}) , J):= A^{\lie(\mathbb{P}_{\beta^{}})}/(A\mu_{\hbar}(J))^{\lie(\mathbb{P}_{\beta^{}})}, 
$$
which is a quotient by an $\hbar$-adically closed ideal. 
The algebra $\Red(A, \lie(\mathbb{P}_{\beta^{}}), J)$ is a deformation of the function algebra on $\Red(T^*(F^{\bullet}Rep(Q^{\dagger},\beta^{\dagger})), \mathbb{P}_{\beta^{}},\mathcal{O})$, but this deformation does not need to be flat. If it is indeed flat, then we are able to conclude that reduction commutes with quantization; 
we leave it as an open problem to find the condition on which reduction and quantization commute for nonreductive group equivariant theory on filtered quiver subrepresentations.

\section{Rational Cherednik algebras and noncommutative deformations of the Hilbert scheme}\label{subsection:rational-Cherednik-algebras}

Gan-Ginzburg quantizes the Hamiltonian reduction of their moment map in their reductive algebra setting in \cite{MR2210660} to obtain the rational Cherednik algebra. In this section, we investigate the quantization of the Hamiltonian reduction in the filtered quiver representation setting. 

\subsection{Symplectic reflection algebras}
We begin with an introduction to symplectic reflection algebras. 

Let $V$ be an $n$-dimensional complex vector space. 
A pseudo reflection is an invertible linear transformation $g\in GL(V)$ of finite order such that the subspace of $V$ invariant under $g$ has dimension $n-1$. 
Pseudo reflections are also known as complex reflections, an invertible linear transformation of $V$ of finite order that fixes a complex hyperplane pointwise. 
Thus, a complex reflection group is a finite group generated by pseudo (complex) reflections. 

\begin{theorem}[\cite{MR0072877}, \cite{MR0234953}, \cite{MR0059914}]
The following are equivalent: 
\begin{enumerate}
\item $V/G$ is smooth, 
\item $\mathbb{C}[V]^{G}$ is a polynomial algebra on $\dim V$ generators, i.e., $\spec(\mathbb{C}[V]^G)\cong \mathbb{A}^{\dim V}$, 
\item $G=(G,V)$ is a complex reflection group. 
\end{enumerate}
\end{theorem}

Now let $l/p\in \mathbb{Z}$. 
We define $G(l,p,n)$ to be the group of $n\times n$ monomial matrices whose nonzero entries are $l$-th root of unity and so that the product of the nonzero entries is an $l/p$-th root of unity.

\begin{example} 
Consider $G(l,1,n)\cong (\mathbb{Z}/l\mathbb{Z})\wr S_n = (\mathbb{Z}/l\mathbb{Z})^n \rtimes S_n$. 
Then $G(1,1,n)\cong S_n$ and $G(2,1,n)\cong B_n$, the Weyl group of type $B$. 
So for $S_n$ acting on $\mathbb{C}[x_1,\ldots, x_n]$ by permuting the indices of the generators, 
we have $\mathbb{C}[x_1,\ldots, x_n]^{G(1,1,n)} \cong \mathbb{C}[\Sigma_1,\ldots,  \Sigma_n]$, where $\Sigma_i = \displaystyle{\sum_{1\leq j_1<\ldots < j_i\leq n}} x_{j_1}\cdots x_{j_i}$ (symmetric polynomials), and more generally, 
$\mathbb{C}[x_1,\ldots, x_n]^{G(l,1,n)}\cong \mathbb{C}[f_1,\ldots, f_n]$, where $f_i = \Sigma_i(x_1^l,\ldots, x_n^l)$, symmetric polynomials evaluated at $x_1^l,\ldots, x_n^l$. 
\end{example}

Since the orbit space $V/G$ does not need to be smooth in general, we give the construction of a symplectic reflection group. 
Let $(V,\omega)$ be a symplectic vector space, where $\omega:V\times V\rightarrow \mathbb{C}$ is a nondegenerate bilinear symplectic form. 
Define the symplectic linear group $\Sp(V)$ as the set of all $g\in GL(V)$ such that $\omega(gu,gv)=\omega(u,v)$ for all $u,v\in V$, i.e., $\omega$ is invariant under $G$. 
We say $s\in G$ is a symplectic reflection if $\rk(s-\Id)=2$.

\begin{definition}
The triple $(G,V,\omega)$ is a symplectic reflection group if 
\begin{enumerate}
\item $(V,\omega)$ is a symplectic vector space, 
\item $G\leq \Sp(V)$ consists of symplectic transformations of $V$, and 
\item $G$ is generated by the symplectic reflections on $V$. 
\end{enumerate}
\end{definition}

From this point forward, we will let $(G,V,\omega)$ be a symplectic reflection group. 
Let $G$ act on $\mathbb{C}[V]$ via ${}^{g}f(v):= f(g^{-1}v)$ for all $f\in \mathbb{C}[V]$, $v\in V$, and $g\in G$. 
The skew group ring $\mathbb{C}[V]\rtimes G$ is a noncommutative algebra, which is isomorphic to $\mathbb{C}[V]\otimes_{\mathbb{C}} \mathbb{C}G$ as a vector space, that satisfies $g\cdot f = {}^{g}f\cdot g$ for all $f\in \mathbb{C}[V]$ and $g\in G$.  
The center $Z(\mathbb{C}[V]\rtimes G)$ of the skew group ring is isomorphic to the ring $\mathbb{C}[V]^G$ of $G$-invariant functions on $V$.

\begin{definition}\label{defn:symplectic-reflection-algebra}
Let $S$ be the set of symplectic reflections of a symplectic reflection group $(G,V,\omega)$. 
Let $TV^*$ be the tensor algebra $\mathbb{C}\oplus V^*\oplus (V^*\otimes V^*)\oplus\ldots$, and let the symplectic form $\omega_{V^*} =  \omega$ under the identification of $V$ and $V^*$. Let 
$\omega_s=\omega$ on $\im(s-\Id)$ and $0$ on $\ker(s-\Id)$.  
A {\em symplectic reflection algebra} is  
\begin{equation}\label{eq:symplectic-reflection-algebra-defn}
H_{t,\mathbf{c}} := H_{t,\mathbf{c}}(G)= TV^* \rtimes G\big/\langle 
[u,v]-t\omega_{V^*}(u,v)1_G+2\sum_{s\in S} \mathbf{c}(s)\omega_s(u,v)\cdot s: u,v\in V^*
\rangle,  
\end{equation}
where 
 $t\in \mathbb{C}$ and $\mathbf{c}:S\rightarrow \mathbb{C}$ is a $G$-conjugacy invariant function, where 
$\mathbf{c}(s)=\mathbf{c}(gsg^{-1})$ for all $s\in S$ and $g\in G$. 
\end{definition}

It is a classical result that  $H_{t, \mathbf{c}}$ are Poincar\'e-Birkhoff-Witt (PBW) deformations of the skew group ring $\mathbb{C}[V]\rtimes G=\Sym(V^*)\rtimes G$ (\cite{MR1881922}, Theorem 1.3). 
That is, letting generators of $G$ have degree $0$ and the generators of $V^*$ have degree $1$, we have a natural filtration $F^{\bullet}$ of $H_{t,\mathbf{c}}$ such that the associated graded $\gr_{F^{\bullet}}(H_{t,\mathbf{c}})$ is isomorphic to $\mathbb{C}[V]\rtimes G$ as algebras. 
So symplectic reflection algebras are deformations of a skew group ring, and since 
$H_{\lambda t,\lambda \mathbf{c}} \cong H_{t,\mathbf{c}}$ where $\lambda \in \mathbb{C}$ is a nonzero scalar, we focus on the cases when $t=0$ or $1$. 
It is straightforward to check that $H_{0,\mathbf{0}}\cong \mathbb{C}[V]\rtimes G$.

Now let $e=|G|^{-1}\sum_{g\in G}g$, the trivial idempotent in the group ring 
$\mathbb{C}G$.  Then 
\[ 
\mathbb{C}[V]^G   \stackrel{\simeq}{\longrightarrow}  
e(\mathbb{C}[V]\rtimes G)e, \mbox{ where } f \mapsto efe . 
\] 
The spherical subalgebra of $H_{t,\mathbf{c}}$ is defined to be 
\[ 
U_{t,\mathbf{c}} := eH_{t,\mathbf{c}} e. 
\] 
By PBW, the graded spherical subalgebra $\gr_{F^{\bullet}}U_{t,\mathbf{c}}$ is isomorphic to the ring $\mathbb{C}[V]^G$ of invariants as algebras. 
The center of the spherical subalgebra is $\mathbb{C}$ for $t\not=0$, while $U_{t,\mathbf{c}}$ is commutative for $t=0$. 
In fact, Satake isomorphism gives the result that $Z(H_{t,\mathbf{c}})\cong Z(U_{t,\mathbf{c}})$. 
It follows that $Z(H_{t,\mathbf{c}})=\mathbb{C}$ for $t\not=0$ while  
$Z(H_{0,\mathbf{c}})=U_{0,\mathbf{c}}$. 

The variety $X_{\mathbf{c}}(G):= \spec U_{0,\mathbf{c}}=\spec (Z(H_{0,\mathbf{c}}))$ is known as a generalized Calogero-Moser space, which is smooth if and only if the dimension of any simple $H_{0,\mathbf{c}}$-module equals $|G|$.

\begin{definition}\label{defn:symplectic-resolution}  
Let $(V/G)_{\sm}$ be the smooth locus of $V/G$. 
A symplectic resolution of $V/G$ is a resolution of singularities $\pi:X\rightarrow V/G$ such that a symplectic form $\omega_X$ satisfying 
\[ 
\pi^*(\omega_{(V/G)_{\sm}}) = \omega_X|_{\pi^{-1}((V/G)_{\sm})}
 \] 
exists on $X$. 
\end{definition} 
A symplectic resolution induces a symplectic isomorphism $\pi|_{\pi^{-1}((V/G)_{\sm})}:\pi^{-1}((V/G)_{\sm})\stackrel{\simeq}{\rightarrow} (V/G)_{\sm}$. 
An interesting result relating orbit spaces and symplectic reflection algebras is the following: 

\begin{theorem}[\cite{MR2065506}, \cite{MR2426349}]
The orbit space $V/G$ admits a symplectic resolution if and only if $X_{\mathbf{c}}(G)$ is smooth for a generic $\mathbf{c}$. 
\end{theorem}
Symplectic resolutions do not need to be unique, but one of many interesting properties about these resolutions is that they are semismall. Note that the Springer resolution is a symplectic resolution and it is semismall: 
\[ 
\dim \widetilde{\mathcal{N}}\times_{\mathcal{N}} \widetilde{\mathcal{N}} \leq \dim \widetilde{\mathcal{N}}. 
\]

\subsection{Rational Cherednik algebras}
Let $V=\mathfrak{h}\oplus \mathfrak{h}^*$ with basis $y_1,\ldots, y_n$ for $\mathfrak{h}$ and $x_1,\ldots, x_n$ for $\mathfrak{h}^*$. Then $V$ has a symplectic form $\omega :V\times V\rightarrow \mathbb{C}$. A standard symplectic form $\omega_V$ is defined to be $\omega_V(y \oplus x , y' \oplus x' ) = x' (y) -x(y')$, where $y, y'\in \mathfrak{h}$ and $x, x'\in \mathfrak{h}^*$. 
Let $\Gamma$ be the image of a complex reflection group $G$ in $GL(\mathfrak{h})\times GL(\mathfrak{h}^*)$, i.e., for $g\in G$, $g(y_i,x_j)= (gy_i,gx_j)$, which preserves the symplectic form. 
Then $s\in \Gamma$ is a symplectic reflection if and only if $s$ is a complex reflection in $G$.  
If $s$ is a complex reflection, then $\rk(s-\Id)=1$ but if $s$ is a symplectic reflection, then $\rk(s-\Id)=2.$

Rational Cherednik algebras are symplectic reflection algebras associated to the triple $(\mathfrak{h}\oplus \mathfrak{h}^*,\omega,W)$, where $W$ is a complex reflection group (also the Weyl group of $GL(\mathfrak{h})$). So $W$ acts diagonally on $\mathfrak{h}\oplus \mathfrak{h}^*$. 
 Let $S$ be the set of all symplectic reflections in $W$ and let $H_s$ be the reflecting hyperplane of $s$, $\alpha_s\in \mathfrak{h}^*$ such that the kernel of $\alpha_s$ is $H_s$, and $\alpha_s^{\vee}\in \mathfrak{h}$ such that $(\alpha_s^{\vee},\alpha_s)=1-\det(s)$. 
Let $\mathbf{c}:S\rightarrow \mathbb{C}$ be the conjugacy invariant function as before. 
We simplify Equation~\eqref{eq:symplectic-reflection-algebra-defn} for rational Cherednik algebras as follows: 
\[  
H_{t,\mathbf{c}} = TV^*\rtimes W/\langle [x_i,x_j]=0, [y_i,y_j]=0,  [y,x]=t(y,x) -2\sum_{s\in S}\dfrac{\mathbf{c}(s)}{1-\det(s)} (y,\alpha_s)(\alpha_s^{\vee},x)s  \rangle,
\] 
where $t\in \mathbb{C}$, $x,x_i,x_j\in \mathfrak{h}^*$ and $y,y_i,y_j\in \mathfrak{h}$.

In type $A$, the presentation of rational Cherednik algebras is given as follows: 
let $\mathfrak{h}=\mathbb{C}^n$. Then $W= S_n$, the symmetric group on $n$ letters. Then $S$ becomes the set of transpositions $s_{ij} = (i\:\:j)$. 
Then 
\begin{equation*}
\begin{aligned}
\sigma\cdot x_i &= x_{\sigma(i)}\sigma ,  \\ 
\alpha_{s_{ij}} &= x_i -x_j,    \\ 
[x_i,x_j]&=0 ,  \\ 
[y_i,x_j] &=\mathbf{c}s_{ij} \mbox{ for all }1\leq i < j\leq n \\ 
\end{aligned}
\hspace{8mm}
\begin{aligned}
\sigma \cdot y_i &= y_{\sigma(i)}\sigma \mbox{ for all } \sigma \in S_n, \\ 
\alpha_{s_{ij}}^{\vee} &= y_i-y_j \mbox{ for } 1\leq i< j\leq n, \\ 
[y_i,y_j] &= 0 \mbox{ for all }i < j,  \\ 
[y_i, x_i] &= t - \mathbf{c} \sum_{j\not= i} s_{ij} \mbox{ for all } 1\leq i\leq n, \\ 
\end{aligned}
\end{equation*}
where $t,\mathbf{c}\in \mathbb{C}$. 
We have an embedding of the spherical subalgebra $eH_{t,\mathbf{c}}e\subseteq H_{t,\mathbf{c}}$ in the rational Cherednik algebra, which in turn contains the following two subalgebras: 
\[ (\Sym\mathfrak{h})^{S_n} =\mathbb{C}[y_1,\ldots, y_n]^{S_n}\hookrightarrow 
eH_{t,\mathbf{c}}e, \hspace{4mm}
\mathbb{C}[\mathfrak{h}]^{S_n}=\mathbb{C}[x_1,\ldots, x_n]^{S_n} \hookrightarrow eH_{t,\mathbf{c}}e \hspace{2mm}\mbox{ via }
\hspace{2mm}
a\mapsto a\cdot e = e\cdot a, 
\] 
where the algebras $(\Sym\mathfrak{h})^{S_n} $ and $\mathbb{C}[\mathfrak{h}]^{S_n}$ generate $eH_{t,\mathbf{c}}e$ as an algebra.   
Furthermore, we have 
\[ 
\gr eH_{t,\mathbf{c}}e = \mathbb{C}[\mathfrak{h}\times \mathfrak{h}^*]^{S_n}
\] 
by PBW. 

Let $\mathcal{D}(\mathfrak{b}\times \mathbb{P},c)$ be the algebra of $c$-twisted differential operators on $\mathfrak{b}\times \mathbb{P}$ and let $\mathfrak{b}_c := \im(\mathfrak{b}\cap \mathfrak{sl}(V)\rightarrow \mathcal{D}(\mathfrak{b}\times \mathbb{P},c))$. 
We end this section with a conjecture analogous to the classical Harish-Chandra homomorphisms. 
\begin{conjecture}\label{conj:Cherednik-algebras-diff-operators}
There exists a subalgebra $H_{1,\mathbf{c}}'\subseteq H_{1,\mathbf{c}}$ such that
\[ 
\Big(\mathcal{D}(\mathfrak{b}\times \mathbb{P},c)/\mathcal{D}(\mathfrak{b}\times \mathbb{P},c)\cdot \mathfrak{b}_c\Big)^{\ad \mathfrak{b}_c} \stackrel{\simeq}{\longrightarrow} e H_{1,\mathbf{c}}' e
\]  
and 
\[ 
\gr \Big(\mathcal{D}(\mathfrak{b}\times \mathbb{P},c)/\mathcal{D}(\mathfrak{b}\times \mathbb{P},c)\cdot \mathfrak{b}_c\Big)^{\ad \mathfrak{b}_c} \stackrel{\simeq}{\longrightarrow} \mathbb{C}[V]^H,  
\]  
where $H=B\cap SL(V)$.  
\end{conjecture}
We leave it as a part of our near future work to investigate the quantization of the Hamiltonian reduction in the filtered quiver representation setting. 



\section{Future directions}\label{section:applications-future-work} 
The coordinate ring $\mathbb{C}[Rep(Q,\beta)]$ has two gradings. 
One way is called {\em $Q_1$-grading}, where the ring is graded by $\mathbb{Z}^{Q_1}$ 
since the quiver variety $Rep(Q,\beta)$ $=$ $ \displaystyle{\bigoplus_{a\in Q_1}M_{\beta(ha)\times \beta(ta)}(\mathbb{C})}$ is a product of matrices.  The second way is called {\em $Q_0$-grading}, where the ring is graded by $\mathbb{Z}^{Q_0}$. To explain this further, 
there exists a natural action of $GL_{\beta}(\mathbb{C}) := \mathbb{G}_{\beta} =
\displaystyle{\prod_{i\in Q_0}GL_{\beta_i}(\mathbb{C})}$ on $Rep(Q,\beta)$ which induces an action on the ring $\mathbb{C}[Rep(Q,\beta)]$. 
So $\displaystyle{\prod_{i\in Q_0}\mathbb{C}^*}$ acts on $\mathbb{C}[Rep(Q,\beta)]$ via the characters of the group, where 
$\mathbb{C}^*=\mathbb{C}\setminus \{ 0\}$. Thus, we can decompose the ring as a direct sum of weight spaces for the action of 
$\displaystyle{\prod_{i\in Q_0}\mathbb{C}^*}$. Let $\mathbb{C}[Rep(Q,\beta)]^{GL_{\beta}(\mathbb{C}),\bullet}$ $:=$ 
$\displaystyle{\bigoplus_{\chi} \mathbb{C}[Rep(Q,\beta)]^{GL_{\beta}(\mathbb{C}),\chi}}$, where $\chi$ is a character of $GL_{\beta}(\mathbb{C})$.  Then polynomials $f\in \mathbb{C}[Rep(Q,\beta)]^{GL_{\beta}(\mathbb{C}),\bullet}$ 
are homogeneous with respect to the $Q_0$-grading.

A polynomial $f\in \mathbb{C}[Rep(Q,\beta)]$ is an invariant polynomial if $g.f=f$ for all $g\in GL_{\beta}(\mathbb{C})$, and the polynomial $f$ is $\chi$-semi-invariant if $g.f=\chi(g)f$ for all $g\in GL_{\beta}(\mathbb{C})$, where $\chi:GL_{\beta}(\mathbb{C})\longrightarrow\mathbb{C}^*$ is a group homomorphism. 
Semi-invariants under the $GL_{\beta}(\mathbb{C})$-action are invariants for $SL_{\beta}(\mathbb{C})$ $:=$ 
$\displaystyle{\prod_{i\in Q_0}SL_{\beta_i}(\mathbb{C})}$-action and 
$SL_{\beta}(\mathbb{C})$-invariant polynomials that are homogeneous with respect to the $Q_0$-grading are also semi-invariant (for some $\chi$) for the $GL_{\beta}(\mathbb{C})$-action.  
Therefore, $\mathbb{C}[Rep(Q,\beta)]^{GL_{\beta}(\mathbb{C}),\bullet} \cong  \mathbb{C}[Rep(Q,\beta)]^{SL_{\beta}(\mathbb{C})}$. 
In the literature, one writes 
$SI(Rep(Q,\beta))$ to mean $\mathbb{C}[Rep(Q,\beta)]^{SL_{\beta}(\mathbb{C})}$. 
Earlier works in the study of invariants of quiver representations include \cite{MR557581}, \cite{MR1113382} and \cite{MR1162487}, with techniques given in 
\cite{MR1758750},  
\cite{MR1825166}, and \cite{MR1908144}.

Other future directions include taking the study of filtered representations to construct GIT quotients to study wall-crossing under the variation of various characters.  
Just as the geometry of $\mathfrak{p}/P$ or $\mathfrak{p}/U$ is interesting, one could generalize this space to $F^{\bullet}Rep(Q,\beta)/\mathbb{P}_{\beta}$, $F^{\bullet}Rep(Q,\beta)/\mathbb{U}$, or their corresponding (quantum) Hamiltonian reduction setting, generalizing the Springer resolution  $T^*(G/B)\twoheadrightarrow \mathcal{N}$. 
The study of the Schubert varieties, vector bundles, or derived categories of coherent sheaves on these varieties 
\[
\xymatrix@-1pc{
\widetilde{\mathcal{N}} \ar@{^{(}->}[rr]  \ar@{->>}[dd] &  & \widetilde{\mathfrak{g}} \ar@{->>}[dd] \\ 
& & \\ 
\mathcal{N} \ar@{^{(}->}[rr] & & \mathfrak{g} \\ 
}
\] 
is rich with many connections to quiver Hecke (KLR) algebras, Hochschild homology of Soergel bimodules, Khovanov-Rozansky homology of a torus knot, and modular representation theory.








   

%



\bibliography{KLR-algebras-Im}

\appendix

\end{document}